\documentclass{amsart}
\usepackage{euscript}
\usepackage{amssymb,amsmath}
\usepackage{tikz}
\usetikzlibrary{arrows.meta, decorations.markings}
\usepackage{booktabs}
\usepackage[final]{hyperref}
\allowdisplaybreaks
\def\twobytwo#1#2#3#4{\bigl(\begin{smallmatrix}#1&#2\\#3&#4\end{smallmatrix}\bigr)}
\def\rmA{\mathrm{A}}
\def\rmB{\mathrm{B}}
\def\rmC{\mathrm{C}}
\def\rmD{\mathrm{D}}
\def\SA{S_{\rmA}}

\def\SD{S_{\rmD}}
\def\sA{s_{\rmA}}
\def\sB{s_{\rmB}}
\def\sC{s_{\rmC}}
\def\sD{s_{\rmD}}

\def\sigmaA{\sigma_{\rmA}}
\def\sigmaB{\sigma_{\rmB}}
\def\sigmaC{\sigma_{\rmC}}
\def\sigmaD{\sigma_{\rmD}}
\def\tauA{\tau_{\rmA}}
\def\tauB{\tau_{\rmB}}
\def\tauC{\tau_{\rmC}}
\def\tauD{\tau_{\rmD}}
\def\Htilde{\widetilde{\H}}

\def\betabar{\bar{\beta}}
\def\deltabar{\bar{\delta}}
\def\Xst{X^{\mathrm{st}}}
\def\Yst{Y^{\mathrm{st}}}
\def\ftilde{\tilde{f}}

\def\chihat{\hat{\chi}}

\def\cpone{\C P^1}
\def\PGDM{\PG_{\!\mathrm{DM}}}
\def\GDM{\Gamma_{\!\mathrm{DM}}}
\def\PGDMsw{\PG_{\!\mathrm{DM}}^{\mathrm{\,sw}}}
\def\BDM{\BB^9_{\mathrm{DM}}}
\def\HDM{\H_{\mathrm{DM}}}
\def\LDM{L_{\mathrm{DM}}}
\def\Art{\mathop{\rm Art}\nolimits}
\def\Br{\mathop{\rm Br}\nolimits}
\def\Cox{\mathop{\rm Cox}\nolimits}
\def\Atilde{\tilde{A}}

\def\semidirect{\rtimes}

\let\iso\cong

\def\Aut{\mathop{\hbox{\rm Aut}}\nolimits}
\def\Out{\mathop{\hbox{\rm Out}}\nolimits}

\def\M{\EuScript{M}}

\def\reverse{\mathop{\hbox{\rm reverse}}\nolimits}

\def\cong{\equiv}
\def\orb{{\rm\scriptstyle orb}}
\def\piorb{\pi_1^\orb}

\def\PGL{\mathop{\rm PGL}\nolimits}

\def\w{\omega}

\def\zbar{\bar{z}}

\def\wbar{\bar{\w}}
\def\omegabar{\bar{\omega}}

\def\E{\EuScript{E}}
\def\Z{\mathbb{Z}}
\def\Q{\mathbb{Q}}
\def\F{\mathbb{F}}

\def\C{\mathbb{C}}

\def\H{\EuScript{H}}
\def\tensor{\otimes}
\def\sset{\subseteq}
\def\G{\Gamma}
\def\PG{P\G}
\def\thetabar{\bar{\theta}}

\def\ip#1#2{\langle#1\,{|}\,#2\rangle}

\def\bigset#1#2{\bigl\{#1\bigm|#2\bigr\}}

\def\spanof#1{\langle#1\rangle}

\def\U{{\rm U}}
\newcommand{\abs}[1]{\lvert #1 \rvert}
\newcommand{\BB}{\mathbb{B}}

\newcommand{\op}[1]{\operatorname{#1}}

\newtheorem{theorem}{Theorem}[section]
\newtheorem{lemma}[theorem]{Lemma}

\theoremstyle{remark}
\newtheorem{numberedremark}[theorem]{Remark}

\newtheorem*{remark}{Remark}
\newtheorem*{remarks}{Remarks}

\numberwithin{figure}{section}
\numberwithin{equation}{section}
\begin{document}
\title{The Deligne-Mostow 9-ball, and the monster}
\author{Daniel Allcock}
\thanks{Allcock supported by Simons
Foundation Collaboration Grant 429818.  Basak supported by
Simons Foundation Collaboration Grant 637005.}
\address{Department of Mathematics\\University of Texas, Austin}
\email{allcock@math.utexas.edu}
\urladdr{http://www.math.utexas.edu/\textasciitilde allcock}
\author{Tathagata Basak}
\address{Department of Mathematics, Iowa State University, Ames IA, 50011.}
\email{tathastu@gmail.com}
\urladdr{https://orion.math.iastate.edu/tathagat/}
\date{September 8, 2020}

\begin{abstract}
    The ``monstrous proposal'' of the first author is that 
    the quotient of a certain $13$-dimensional
    complex hyperbolic braid group,
    by the relations that its natural generators have order~$2$, is 
    the ``bimonster'' $(M\times M)\semidirect2$.  Here $M$ is the monster
    simple group.  We prove that this
    quotient is either the bimonster or~$\Z/2$.  
    In the process, we give new information about the isomorphism found by
    Deligne-Mostow, between
    the moduli space of $12$-tuples in~$\C P^1$ and a quotient 
    of the complex $9$-ball.  Namely, we identify
    which loops in the $9$-ball quotient correspond to the standard
    braid generators.
\end{abstract}

\maketitle
%

\section{Introduction}
\label{sec-introduction}

\noindent
This paper continues our pursuit of a conjectured relationship between
complex hyperbolic geometry and the largest sporadic finite simple group,
the monster~$M$ \cite{Basak-bimonster-1}\cite{Basak-bimonster-2}\cite{Allcock-Y555}\cite{AB-braidlike}\cite{AB-26gens}.  In \cite{Allcock-monstrous},
the first author conjectured a ``monstrous proposal'':
the quotient of a
certain complex hyperbolic braid group~$G$,
by some natural relations $S=1$,
coincides with the semidirect product
$(M\times M)\semidirect2$.
This group is called the bimonster, and the $\Z/2$ exchanges the monster
factors in
the obvious way. 
In this paper we prove the conjecture up to one other possibility: 
$G/S$ is 
either the bimonster or~$\Z/2$.

Before describing $G$, we 
recall the source of the ``complex hyperbolic
braid group'' terminology.
The classical $n$-strand braid group can be described as follows
\cite{Fox-Neuwirth}.
Start with $\C^n$, remove the mirrors (fixed-point sets) of the
reflections in the symmetric group $S_n$, quotient what remains
by the action of 
$S_n$, and then take the fundamental group.
There are many generalizations of this construction, got by
replacing $S_n$ by 
some
other reflection group 
and (optionally) replacing $\C^n$ by some other space.
Three examples are  Artin groups \cite{van-der-Lek}\cite{Charney-Davis},
the braid groups of the finite
complex reflection groups \cite{Bessis}, and the fundamental groups of
certain discriminant complements arising 
in singularity theory \cite{Brieskorn}\cite{Looijenga-triangle-singularities}.  
Complex reflections play a special role because
their mirrors  have real
codimension two: removing them leaves
a space which is connected but not simply-connected.

We apply this general construction to
a specific group $P\G$,
generated by complex reflections
acting on complex hyperbolic space $\BB^{13}$.
We will give a precise description of $P\G$ in 
section~\ref{sec-background-conventions-notation}, but the
details are not important
yet.  We define $\H$ as the union of the
mirrors of the complex reflections in $P\G$.  
The braid group we will study
is the orbifold fundamental group 
$$G=\piorb((\BB^{13}-\H)/P\G).$$  
A generic point of $\H$ has stabilizer generated by a
single triflection (order~$3$ complex reflection), 
so its image in $\BB^{13}/P\G$ has a neighborhood
which is smooth and 
whose intersection with $\H/P\G$ is a smooth hypersurface.
By a \emph{meridian} we
mean a small loop in $(\BB^{13}-\H)/\PG$
encircling this hypersurface once positively, or any loop
freely homotopic to such a loop.
The analogous definition for the 
classical braid group gives the conjugacy class containing
the standard generators.  
It is well-known that killing the squares of the standard generators
reduces the $n$-strand braid group to 
the symmetric group~$S_n$.
In this paper we study
the analogous quotient for~$G$:  

\begin{theorem}
    \label{ThmMain}
    The complex hyperbolic braid group
    $G=\piorb((\BB^{13}-\H)/P\G)$, modulo
    the subgroup~$S$ generated by the squares of all meridians,
    is isomorphic to either the bimonster or $\Z/2$.
\end{theorem}

Ruling out the
case $G/S\iso\Z/2$ would prove the monstrous
proposal.  

\smallskip
Although we do not know a presentation for~$G$,
we do know that it is a quotient of 
certain Artin group.    The Artin group $\Art(\Delta)$ 
of a graph $\Delta$ is defined as follows. 
It has
one generator for each node of the graph.  Its
defining relations are that two of these generators $x,y$
braid ($xyx=yxy$) or commute ($xy=yx$) according to
whether their nodes are joined or not.  
The associated Coxeter group $\Cox(\Delta)$ is defined
by imposing the additional relations that the squares
of the standard
Artin generators vanish.

We will write $P^2\F_3$ for the incidence graph of
the $13$ points and $13$ lines of the projective
plane over the field $\F_3$ 
of order~$3$.  The main result of \cite{AB-26gens}
is that $G$ is a quotient of
$\Art(P^2\F_3)$, with the standard Artin generators
mapping to meridians.

We will prove theorem~\ref{ThmMain} by finding that certain
elements of $\Art(P^2\F_3)$ vanish
in~$G$.  It follows that $G/S$ is the
quotient of the Coxeter group $\Cox(P^2\F_3)$
by some specific relations and possibly some additional
(unknown) relations.  The relations
we find are 
enough to reduce 
$\Cox(P^2\F_3)$ to the
bimonster.  
This relies on a result of Conway and Simons \cite{Conway-Simons}.
It follows that $G/S$
is a quotient of the bimonster.
By the simplicity of~$M$, the only quotients are the
bimonster itself, $\Z/2$ and the trivial group.  A short
automorphic forms argument
rules out the trivial case (lemma~\ref{LemGNontrivial}).

\medskip
While our result gives an upper bound for $G/S$, 
Looijenga \cite{Looijenga-monster} has found a construction
that suggests 
that $G/S$ is no smaller than the bimonster.  If $G/S$ were the bimonster,
then the covering space of 
$(\BB^{13}-\H)/\PG$ corresponding to~$S$ would have the bimonster as its
deck group.  Then, every open subset of $(\BB^{13}-\H)/\PG$ would
have an orbifold cover with this deck group, namely its preimage.  
Looijenga has found an open subset over which this hoped-for cover exists
unconditionally, and is connected.  Furthermore, 
this subset is large enough to resemble $(\BB^{13}-\H)/\PG$
itself.  

\medskip
Part of the appeal of the monstrous proposal is that it
suggests a geometric interpretation of
the Conway-Simons relations that collapse $\Cox(P^2\F_3)$
to the bimonster.   For each $12$-gon in $P^2\F_3$, these kill the
translation subgroup of the corresponding affine Coxeter group
$\Cox(\Atilde_{11})\iso\Z^{11}\semidirect S_{12}$, leaving just
 the symmetric group 
$\Cox(A_{11})\iso S_{12}$.  If $\tau_0,\dots,\tau_{11}$ are the
generators of $\Cox(\Atilde_{11})$, corresponding to the
nodes of the $12$-gon in cyclic order, then one can achieve
this by imposing the relations that $\tau_j\tau_{j+1}\cdots\tau_{j+10}$
is independent of~$j$.
These are exactly the relations required to collapse $\Art(\Atilde_{11})$
to the $12$-strand braid group.  (See theorem~\ref{ThmBraidGroups}.) 
This suggests that the Conway-Simons
relations are the shadows of some relations that hold
in~$G$, even before imposing the relations that meridians 
square to~$1$.   

This dovetails with the
fact the $12$-strand braid group is prominent in our braid group~$G$.
This happens because $\BB^{13}$ contains a 9-ball $\BDM$,
whose setwise stabilizer  acts on it as the group $\PGDM$ that
Deligne and Mostow used to uniformize the moduli
space of unordered $12$-tuples in $\cpone$.  
That is, writing $\HDM$ for the hyperplane arrangement in $\BDM$ got
by restricting~$\H$, $(\BDM-\HDM)/\PGDM$
is the moduli space $\M_{12}$ of unordered $12$-tuples of distinct points.
The connection to braids is that
$\piorb(\M_{12})\iso\Br_{12}(\cpone)/(\Z/2)$.
Our strategy will be
to choose a tubular neighborhood $U$ of $\BDM$, invariant
under the setwise stabilizer $\PGDMsw$ of $\BDM$.
It develops that $(U-\H)/\PGDMsw$ fibers over $(\BDM-\HDM)/\PGDM$, 
with
fiber that can be understood by work of Orlik and Solomon.
This leads to  an exact sequence
\begin{equation}
    \label{EqIntroExactSeq}
1\to\Br_5\to \piorb((U-\H)/\PGDMsw) \to \piorb(\M_{12})\to 1
\end{equation}
We work out the details of this (nonsplit)
group extension.  This yields several explicit relations in the Artin 
generators of~$G$.  This is done in sections \ref{SecNeighborhood}--\ref{SecMain}.
After the preliminary sections \ref{sec-background-conventions-notation}
and \ref{SecBraidGroups}, it is possible
to jump to section~\ref{SecNeighborhood}, armed only with the statements 
of theorems \ref{ThmBasepointChange} and~\ref{ThmDeligneMostow}.

Heckman \cite{Heckman} gave an argument that $G$ is a quotient of the
bimonster, but unfortunately it had a gap.  Our section~\ref{SecChange}
can be used to bridge that gap, bypassing some of the details
that we develop in this paper.  See remark~\ref{RkHeckmanArgument}
for Heckman's argument.   The largest difference from our work
is that he seeks extra relations in $G/S$, rather than
in $G$ itself, as we do.
Our fibration argument
also applies to some other ball quotients with distinguished 
subballs, yielding analogues
of the exact sequence \eqref{EqIntroExactSeq}.
For example, 
we hope to  give a new presentation of the orbifold fundamental
group of the moduli space space of smooth cubic surfaces, in
terms of  the Artin group of the Petersen graph.  

We are very grateful to Gert Heckman and Eduard Looijenga 
for conversations and
correspondence.

\section{Background, conventions, notation}
\label{sec-background-conventions-notation}

As far as possible we maintain notations and conventions used in 
\cite{AB-braidlike}, \cite{AB-26gens}.
For convenience 
we review the most important ones.  
We follow the convention that a
hermitian form $\ip{\cdot}{\cdot}$ is linear in its first variable and conjugate-linear
in its second.  
The norm of a vector~$x$ means $\ip{x}{x}$.
Often we use the same symbol for a vector and the point of
projective space it represents.  
A superscript  ${}^\perp$ indicates an orthogonal complement, the 
precise meaning depending on context.  It might mean a sublattice,
a linear subspace, or a complex hyperbolic space inside a larger
complex hyperbolic space.
When describing groups, we often use \cite{ATLAS} names,
and the numeral `2' as shorthand
for $\Z/2$, for example the semidirect product $\PGL_3(\F_3)\semidirect(\Z/2)$
becomes $L_3(3)\semidirect2$.

The only material here, not already in \cite{AB-26gens},
is the construction of lattices from graphs (section~\ref{subsec-eisenstein-lattices}),
the notation $L_4$ for the Eisenstein $E_8$ lattice, and the
Deligne-Mostow lattice (sections~\ref{subsec-DM-lattice}).

\subsection{Complex hyperbolic space}
\label{subsec-complex-hyperbolic-space}
A hermitian form is called Lorentzian if it has signature $(n,1)$.
For example, $\C^{n,1}$ denotes $\C^{n+1}$ equipped with the
hermitian form
\begin{equation*}
\ip{x}{y}=-x_0 \bar{y}_0+x_1\bar{y}_1+\dots+x_{n}\bar{y}_{n}.
\end{equation*}
Let $V$ be a complex vector space equipped with a Lorentzian hermitian form.
Then $\BB(V)$ means the set of negative-definite complex
lines in~$V$.  This is
an open ball in $P(V)$.  If $M$ is a subspace of $V$, such that
the restriction of the hermitian form is also Lorentzian, then there
is a natural inclusion $\BB(M)\to\BB(V)$.  
If $M$ has codimension~$1$,
then we sometimes call $\BB(M)$ a hyperplane in $\BB(V)$.
The complex ball $\BB(V)$ has a natural negative curvature metric called
the Bergman metric,
and is sometimes called complex hyperbolic space.  Its 
restriction to 
$\BB(M)$ coincides with the Bergman metric of~$\BB(M)$.
Suppose $v,w$ are negative-norm vectors in~$V$.  Then they represent points
of $\BB(V)$ whose distance is  
\begin{equation}
\label{eq-complex-hyperbolic-metric}
  d(v, w) = \cosh^{-1} \sqrt{  \frac{\abs{   \ip{v}{w}}^2}{v^2\, w^2} } 
\end{equation}
If in addition $\ip{v}{w}$ is real and negative, then the real
line segment joining
$v$ and $w$ in $V$ projects to the
geodesic in $\BB(V)$ between these points.
Similarly, if
$v,s\in V$ have negative and positive norm respectively, then  
\begin{equation}
\label{eq-distance-point-to-hyperplane}
  d\bigl(v, \BB(s^\perp)\bigr) = \sinh^{-1} \sqrt{ -\frac{\abs{   \ip{v}{s}}^2}{v^2\, s^2} } 
\end{equation}

\subsection{Eisenstein lattices}
\label{subsec-eisenstein-lattices}
Let $\omega = e^{2 \pi i/3}$ and $\theta = \omega - \wbar =
\sqrt{-3}$. Let $\E$ be the ring $\Z[\omega]$ of Eisenstein integers.
An Eisenstein lattice $K$ means 
a free $\E$-module equipped with an hermitian form $\ip{\;}{\;} : K \times K \to
\Q(\w)$.
We abbreviate $K\tensor_\E\C$ to $K\tensor\C$, and usually think of
it as containing $K$.  If $K$ is Lorentzian, then we write $\BB(K)$
for $\BB(K\tensor\C)$.
If $K$ is
nondegenerate, then its dual lattice is defined as $K^* = \lbrace
x \in K \otimes \C \colon \ip{x}{k} \in \E \text{\; for all \;} k \in
K \rbrace$.  

Let $\Delta$ be a directed graph without self-loops or multiple edges.
Consider the free $\E$-module $\E  \Delta $ on the vertex set $\Delta$
with basis vectors $\lbrace e_{\alpha} \colon \alpha \in \Delta \rbrace$.
Define an $\E$-valued hermitian form on $\E  \Delta$  by
\begin{equation*}
\ip{e_{\alpha}}{e_{\beta}}=
\begin{cases}
3 & \text{\; if \;} \alpha = \beta \\
\theta &  \text{\; if $\Delta$ has an edge from $\beta$ to~$\alpha$}
     \\
    0 &  \text{\; if $\Delta$ has no edge between $\beta$ and $\alpha$}
\end{cases}
\end{equation*}
If $\Delta$ is a tree, then (by changing the signs of some of the basis
vectors) the isometry class of
this lattice depends only on the undirected graph
underlying $\Delta$.  

\subsection{\texorpdfstring{The Eisenstein $E_8$ lattice and the hyperbolic 
cell.}{The Eisenstein E8 lattice and the hyperbolic cell}}
\label{subsec-E8}
Applying this construction to 
the  Dynkin diagram $A_4$ yields the
$\E$-lattice whose underlying real form, under the
bilinear form $\tfrac{2}{3} \op{Re} \ip{x}{y}$,
is the $E_8$ root lattice.  We called it $E_8^\E$ in
\cite{AB-26gens}, but in this paper we will call it
$L_4$. 
The self-duality of the $E_8$ lattice leads to the property 
$\theta L_4^* = L_4$.  A useful consequence
of this is that if $K$ is an $\E$-lattice 
in which all inner products are divisible by~$\theta$,
then every copy of $L_4$ in~$K$ is an orthogonal direct summand.

The property $\theta K^*=K$ of an Eisenstein lattice~$K$
turns out to characterize~$K$ if $K$ is indefinite
\cite{Basak-bimonster-1}.  
Two
examples are the lattices $L$ and~$\LDM$ below.  Another is the
hyperbolic cell, which means the 
$\E$-span of two null vectors with inner product~$\theta$.

\subsection{\texorpdfstring{The lattice $L$}{The  lattice L}}
\label{t-P2-F3-model-of-L}
The following lattice $L$ plays a central role in this paper.
The quickest way to define it is to
refer to the uniqueness property just stated.
Namely, $L$ is  the unique $\E$-lattice
of signature $(13,1)$ that satisfies $\theta L^*=L$.  The
quickest concrete construction is to define 
$L$ as the sum of a hyperbolic cell and $3$ copies
of $L_4$.  In \cite{AB-braidlike} and \cite{AB-26gens}
we also used a different direct sum description of~$L$,
with the Leech lattice in place of $L_4^3$, but this plays
no role in this paper.  The only $13$-ball we will discuss in this
paper is $\BB(L)$, so we will write $\BB^{13}$ for it. 
All complex hyperbolic geometry in this paper takes place
in $\BB^{13}$. 

We will use the following model of~$L$ for most calculations.
It
was implicit in \cite{Basak-bimonster-1} (eq. 25 in proof of
prop. 6.1) and made  explicit in \cite{Allcock-Y555}.
We write $P^2 \F_3$ for the incidence graph of the finite projective plane 
over $\F_3$. 
It has $26$ vertices, corresponding
to the $13$ points and the
$13$ lines of this finite projective plane.  Whenever $\beta$
is a line, and $\alpha$ is a point on it, we direct the
edge  from the line to the point.
One verifies that the associated rank $26$ lattice (see \ref{subsec-eisenstein-lattices})
has radical of rank $12$.  One may define $L$ as  the quotient by 
the kernel.  Obviously all inner products
are divisible by~$\theta$, ie $L\sset\theta L^*$,
and one can check that this inclusion is equality.

It is possible to label the points and lines of $P^2\F_3$
by $p_1,\dots,p_{13}$ and $l_1,\dots,l_{13}$, such that 
the points on $l_j$ are $p_j$, $p_{j+1}$, $p_{j+3}$ and $p_{j+9}$.
Subscripts here should be read mod~$13$.  
We use the same symbols for the vectors in~$L$ corresponding to
the points and lines.  Because they are roots in the sense
of section~\ref{subsec-roots-in-general}, we call them the point- and line-roots.
Since  points (resp.\ lines) are not joined to each other, the point-roots
(resp.\ line-roots)
are mutually orthogonal.  
One may introduce coordinates
$(x_0;x_1,\dots,x_{13})$ on $\C^{13,1}$ such that 
$p_1=(0;\theta,0,\dots,0)$, 
$l_1=(1;1,1,0,1,0,0,0,0,0,1,0,0,0)$, and
rightward cyclic permutation of the last $13$ coordinates 
increases
subscripts by~$1$.

We write $\Gamma$ for the isometry group of~$L$.    Obviously
it contains the group $L_3(3):=\PGL_3(\F_3)$, permuting the
points and lines of $P^2\F_3$ in the natural way.  
There is additional symmetry.  From 
an incidence-preserving exchange of points with
lines, one can construct an isometry of $L$ that sends the point-roots
to line-roots and the line-roots to negated point-roots. 
Together with scalars and
$L_3(3)$, this generates a subgroup 
of $\Gamma$ whose image in $P\Gamma$ is $L_3(3)\semidirect2$.
We will
use this $L_3(3)\semidirect2$ many times.

\subsection{\texorpdfstring{Roots, mirrors and the 
hyperplane arrangement $\H$}{Roots, mirrors
and the hyperplane arrangement}}

\label{subsec-roots-in-general}
A root of $L$ means a lattice vector of norm~$3$.  
Roots are special because their triflections (complex reflection
of order~$3$) give elements of~$\Gamma$.
Namely, if $s$ is a root, then 
we define $\omega$-reflection in $s$ to be
the isometry of $L\tensor\C$ that fixes $s^{\bot}$ pointwise and
multiplies $s$ by the cube root of unity $\omega$.  A formula is
$
x\mapsto x-(1-\w)\frac{\ip{x}{s}}{s^2}s
$.
Using
$L\sset\theta L^*$, one can show that this preserves $L$.  
Although we don't need it, we remark that $\Gamma$ is generated by the
triflections in the $26$ point- and line-roots;
see \cite{Basak-bimonster-1} or \cite{Allcock-Y555}.

The hyperplane 
$\BB(s^\perp)\sset\BB(L)$ is called the mirror of~$s$, the name
reflecting the fact that it is the fixed-point set of a reflection.
We use the word mirror exclusively for hyperplanes orthogonal to roots.
The union of all
mirrors of $L$ is called $\H$.  This hyperplane arrangement is central
to the paper: our goal is to understand the orbifold fundamental
group of $\bigl(\BB^{13}-\H\bigr)/P\Gamma$.

\subsection{\texorpdfstring{Special points in $\BB^{13}$}{Special points in B13}}
\label{subsec-point-roots-etc}

We have already introduced the point-roots $p_1,\dots,
\discretionary{}{}{}p_{13}$
and line-roots $l_1,\dots,l_{13}$.  We call their
mirrors the point- and line-mirrors.
The $p_i$ resp.\ $l_i$ are mutually orthogonal, and
we write $p_\infty$ resp.\ $l_\infty$
for the point of $\BB^{13}$ orthogonal to
all of them. This turns out
to be represented by the norm~$-3$ vector
$p_\infty=(\thetabar;0,\dots,0)$
resp.\ $l_\infty=(4;1,\dots,1)$.

A convenient basepoint for the orbifold fundamental group
(see section~\ref{subsec-orbifold-fundamental-group})
of $(\BB^{13} - \H)/P\Gamma$ is 
the midpoint $\tau$ of
the geodesic segment joining $p_\infty$ and $l_\infty$.  It is represented by
the vector 
\begin{equation*}
\tau = l_{\infty} + i p_{\infty}=(4+\sqrt3;1^{13})
\end{equation*}
of norm $-6-8\sqrt3$.  The corresponding point of
$\BB^{13}$ is the unique fixed point of $L_3(3)\semidirect2\sset
P\Gamma$.  The mirrors closest to $\tau$ are exactly
the $26$ point- and line-mirrors.  (See
\cite[prop.~1.2]{Basak-bimonster-1}, where $\tau$ was called
$\bar{\rho}$, or \cite[Lemma A.5]{AB-26gens}).
Two consequences of this are that no mirrors pass through $\tau$, and
that $L_3(3)\semidirect2$ is the full $\PG$-stabilizer of $\tau$.  

\subsection{\texorpdfstring{The Deligne-Mostow lattice $\LDM$}{The Deligne-Mostow Lattice}}
\label{subsec-DM-lattice}
The lattice $\LDM$ is both a sublattice of~$L$ and a lower-dimensional
analogue of~$L$.  We start with its role as an analogue.  Consider the
$\Atilde_{11}$ Dynkin diagram (a $12$-gon),
with its edges' orientations alternating.
The corresponding lattice (section~\ref{subsec-eisenstein-lattices}) has $2$-dimensional nullspace,
and $\LDM$ is defined as the quotient by it. 
(See \cite[Sec.\ 5]{Allcock-Inventiones},
which also displays an isometry between $\LDM$ and the sum of a hyperbolic
cell and two copies of~$L_4$.)
$\LDM$ has signature
$(9,1)$ and satisfies $\theta\LDM^*=\LDM$.  We define $\BDM$ as $\BB(\LDM)$,
and $\GDM$ as the
isometry group of~$\LDM$. 
See section~\ref{SecDeligneMostow} for a refinement
of the famous relationship found by Deligne-Mostow,
between $\BDM/\PGDM$
and the moduli space of unordered $12$-tuples in $\cpone$.  
By construction, $\PGDM$ contains a dihedral group $D_{24}$, of
order~$24$, that permutes the vertices of the $12$-gon.  
(Because of the orientations
of edges, one must insert some signs to get isometries from 
permutations, just as when we enlarged $L_3(3)$ to $L_3(3)\semidirect2$.)
We write $\rho$ for the unique fixed point in $\BDM$ of this 
dihedral group.  We define roots as for~$L$, and define $\HDM\sset\BDM$ as
the union of the mirrors of the roots of~$\LDM$.  The mirrors closest 
to~$\rho$
are those whose roots correspond to the nodes of the $12$-gon
(lemma~\ref{LemMirrorsNearRho}).

$\LDM$ appears as a sublattice of~$L$
because the graph $P^2\F_3$ contains a $12$-gon (with suitably oriented
edges).  
Because $\LDM=\theta\LDM^*$ and all inner products in~$L$ are divisible
by~$\theta$, $\LDM$ is a summand, so $\GDM\sset\Gamma$.
We make a particular choice of $12$-gon in section~\ref{SecChange}.
The $(L_3(3)\semidirect2)$-stabilizer of this $12$-gon is the $D_{24}$ just
mentioned.  The projection of $\tau\in\BB^{13}$ to $\BDM$
is the point~$\rho$.  Part of our work will involve
moving a basepoint along $\overline{\rho\tau}$. We prove in lemma~\ref{LemMirrorsMeetingBDM}
that $\HDM$ is the restriction of $\H\sset\BB^{13}$ to $\BDM$.

\subsection{Orbifold fundamental groups.}
\label{subsec-orbifold-fundamental-group}
Let $A$ be a group acting properly discontinuously 
on a path connected manifold $X$.
Choose a base point $b \in X$.
The orbifold fundamental group 
$\piorb(X/A,b)$
 is defined as the following set of equivalence classes of
pairs 
$(\gamma,g)$, where $g\in A$ and $\gamma$ is a path in $X$
from $b$ to $g b$.  One such pair is equivalent to another
one $(\gamma',g')$ if $g=g'$ and $\gamma$ and $\gamma'$ are homotopic
in $X$, rel endpoints.    The group operation is 
\begin{equation*}
    (\gamma,g)\cdot(\gamma',g')=(\gamma \hbox{ followed by } g\circ\gamma' ,g g')
\end{equation*}
Inversion in $\piorb(X/A, b)$ is given by 
$(\gamma,g)^{-1}=(g^{-1}\circ\reverse(\gamma),g^{-1})$.
Projection of $(\gamma,g)$ to $g$ defines a homomorphism $\piorb(X/A,b) \to A$.
It is surjective because $X$ is path-connected.
The kernel is obviously $\pi_1(X, b)$, yielding the exact sequence
\begin{equation}
\label{eq-exact-sequence-on-pi-1}
1\to\pi_1(X,b) \to  \piorb(X/A, b) \to A \to1.
\end{equation}
The local group at~$b$ means the set of $(\gamma,g)\in\piorb(X/A,b)$ 
for which $\gamma$ is homotopic to 
the constant path at~$b$.  This is obviously
the same as the $A$-stabilizer of~$b$.

\subsection{Meridians.}

Meridians are distinguished elements of $\piorb((\BB^{13}-\H)/\PG,b)$ 
or $\piorb((\BDM-\HDM)/\PGDM,b)$,
where
$b\in\BB^{13}-\H$ or $\BDM-\HDM$ is a basepoint.  
If $s$ is a root and $S$ is the $\omega$-reflection in~$s$,
then the corresponding meridian
$M_{b,s}$ is defined as $(\mu_{b,s},S)$ where $\mu_{b,s}$
is the following path.
Let $p$ be the projection of~$b$ into the mirror $s^\perp$.
In this paper, $\overline{bp}$ never meets any other mirror.
Choose a ball around~$p$ small enough to miss all other mirrors,
and let $q$ be a point of $\overline{bp}-\{p\}$ in this ball.
Then $\mu_{b,s}$ is the geodesic $\overline{bq}$, followed by
the positive circular arc in $\BB^1(b,q)$, of angle $2\pi/3$
centered at~$p$, followed by $S(\overline{qb})$.

The extra generality about meridians that we developed in
\cite{AB-braidlike}
is not needed here.  In particular, 
no detours of the sort considered there
are needed.
The only basepoints that we will need are $\rho\in\BDM$ for
$\piorb((\BDM-\HDM)/\PGDM)$,
and a variable point $\sigma\in\overline{\rho\tau}-\{\rho\}\sset\BB^{13}$
when working with $\piorb((\BB^{13}-\H),\PG)$.
Once one of these basepoints is fixed, 
the meridians associated to the point- and line-roots are called the
point- and line-meridians.  No other meridians appear in this paper.

\section{Braid Groups}
\label{SecBraidGroups}

\noindent
In this section we assemble some well-known material about braid groups,
in a form exhibiting dihedral symmetry $D_{2n}$ where $n$ is the number
of strands, which will remain constant in this section.
The $n=12$ case will be key in later sections.  Although our presentations
are slightly different, \cite{Birman} is an excellent general reference.

For a Riemann surface $\Sigma$, we define $X(\Sigma)$ as $\Sigma^n$, 
the space of ordered $n$-tuples in~$\Sigma$.  In this paper
we restrict attention to  the
cases $\Sigma=\C,\C^*,\cpone$. 
The $n$-strand pure braid space of  $\Sigma$ means 
$$
X^\circ(\Sigma)
:=
\bigset{(x_0,\dots,x_{n-1})\in X(\Sigma)}{
    \hbox{$x_j\neq x_k$ whenever $j\neq k$}}
$$
The symmetric group $S_n$ acts freely on $X(\Sigma)$ 
and $X^\circ(\Sigma)$, with 
quotients which will be written
$Y(\Sigma)$ and $Y^\circ(\Sigma)$.
The latter 
is called the $n$-strand braid space of~$\Sigma$, and
its fundamental group is called the $n$-strand braid
group of~$\Sigma$, written $\Br_n(\Sigma)$.
The ordinary braid group $\Br_n$ refers  to the case $\Sigma=\C$.
Setting $\zeta=e^{2\pi i/n}$, 
we will use the
tuple
$T=(1, \zeta,\zeta^2,\dots,\zeta^{n-1})$ as the basepoint
of $X^\circ(\Sigma)$, and its image in $Y^\circ(\Sigma)$
as the basepoint there.
Usually one specifies a loop in $Y^\circ(\Sigma)$ by writing down
a path in $X^\circ(\Sigma)$ from $T$ to one of its $S_n$-images.

We now define specific braids $\rho_j\in\Br_n(\C^*)$, whose subscripts  
should be read modulo~$n$,
by specifying the motion of points of $T$.  
Informally: 
the points beginning at 
$\zeta^{j-1}$
and $\zeta^j$ approach each other, move around each other
in a counterclockwise direction, and then continue 
to $\zeta^j$ and $\zeta^{j-1}$ respectively.  
The remaining points do not move.
For a precise
definition, choose
$r:[0,1]\to[1,\infty)$ continuous
with $r(0)=r(1)=1$ and $r(\frac12)>1$.
Then $\rho_j$ is the element of $\Br_n(\C^*)$ represented by
the path $(x_0(t),\dots,x_{n-1}(t))$ in $X^\circ(\C^*)$, where
\begin{align*}
    x_{j-1}(t)&{}=\zeta^{j-1}r(t)e^{2\pi i t/n}
    \\
    x_{j}(t)&{}=\zeta^{j}r(t)^{-1}e^{-2\pi i t/n}
    \\
    x_{k}(t)&{}=\zeta^k\qquad\qquad\qquad\qquad\hbox{ if $k\neq j-1,j$}
\end{align*}
We define the ``increasing'' and ``decreasing'' words
\begin{equation*}
    I_j{}=\rho_j\rho_{j+1}\cdots\rho_{j+n-2}
    \qquad
    D_j{}=\rho_j\rho_{j-1}\cdots\rho_{j-n+2}
\end{equation*}
(Following our conventions for orbifold fundamental groups in
section~\ref{subsec-orbifold-fundamental-group},
$I_j$ means $\rho_j$ followed by $\rho_{j+1}$,
followed by etc., and similarly for $D_j$.  Because $S_{12}$
acts freely on~$X^\circ$, we have identified elements of 
$\piorb(Y^\circ)=\piorb(X^\circ/S_n)$
with their underlying paths.)
Picture-drawing establishes
\begin{align}
    \label{EqIIncreases}
    I_j\rho_k I_j^{-1}&{}=\rho_{k+1}\hbox{ for all $k\neq j-1,j-2$}
    \\
    \label{EqDDecreases}
    D_j\rho_k D_j^{-1}&{}=\rho_{k-1}\hbox{ for all $k\neq j+1,j+2$}  
\end{align}
The inclusions 
$\C^*\to\bigl(\C^*\cup\{\hbox{$0$ or $\infty$}\}\bigr)\to\cpone$ 
induce homomorphisms
$$
\Br_n(\C^*)
\to
\Br_n\bigl(\C^*\cup\{\hbox{$0$ or $\infty$}\}\bigr)
\to
\Br_n(\cpone).
$$ 
Sometimes we speak of the $\rho_j$ as though they were
elements of these other groups, meaning their images there.

\begin{theorem}[Braid Groups]\leavevmode
    \label{ThmBraidGroups}
    \begin{enumerate}
        \item
            \label{ItemAtildeRelations}
            The subgroup of $\Br_n(\C^*)$ generated by 
            $\rho_0,\dots,\rho_{n-1}$ has defining relations
            $\rho_j\rho_k\rho_j=\rho_k\rho_j\rho_k$ and
            $\rho_j\rho_k=\rho_k\rho_j$, according to whether
            $k\in\{j\pm1\}$ or not.
        \item
            \label{ItemAdjoinZeroRelations}
            Adjoining to \eqref{ItemAtildeRelations} 
            the relations that all 
            $D_j$ coincide yields $\Br_n(\C^*\cup\{0\})$.
            Then $D\rho_k D^{-1}=\rho_{k-1}$ for all~$k$, where
            $D$ is the common image of all $D_j$.
        \item
            \label{ItemAdjoinInftyRelations}
            Adjoining to \eqref{ItemAtildeRelations} 
            the relations that all 
            $I_j$ coincide yields $\Br_n(\C^*\cup\{\infty\})$.
            Then $I\rho_k I^{-1}=\rho_{k+1}$ for all~$k$, where
            $I$ is the common image of all $I_j$.
        \item
            \label{ItemAdjoinBothRelations}
            Adjoining to \eqref{ItemAtildeRelations},
            \eqref{ItemAdjoinZeroRelations} and
            \eqref{ItemAdjoinInftyRelations} the relation
            $ID=1$ yields $\Br_n(\cpone)$.
    \end{enumerate}
\end{theorem}

\begin{remarks}
    The proof of \eqref{ItemAdjoinZeroRelations}
    shows that adjoining to \eqref{ItemAtildeRelations}
    any single relation $D_j=D_{j'}$,
    with $j\neq j'$,
    implies the equality of all~$D_j$. And
    similarly for
    \eqref{ItemAdjoinInftyRelations}.  

    The proof of \eqref{ItemAdjoinBothRelations} shows:
    in the presence of the relations \eqref{ItemAtildeRelations}
    and any
    one relation $I_kD_j=1$, the relations in
    \eqref{ItemAdjoinZeroRelations}
    imply those in \eqref{ItemAdjoinInftyRelations} and vice-versa.
    So $\Br_n(\cpone)$ can be got by adjoining
    two relations to $\Art(\Atilde_{n-1})$, for example
    $I_0=I_1$ and $D_{n-1}I_1=1$.  

    In particular, one could omit either \eqref{ItemAdjoinZeroRelations}
    or \eqref{ItemAdjoinInftyRelations} from the presentation
    \eqref{ItemAdjoinBothRelations}    of $\Br_n(\cpone)$.
    (This omission
    would leave one of $I$ and $D$ undefined, which is 
    why we write $I_kD_j=1$ not $ID=1$.)
    We prefer to keep both \eqref{ItemAdjoinZeroRelations}   
    and \eqref{ItemAdjoinInftyRelations},  
    because 
    in theorem~\ref{ThmPi1ofUmodStabilizer} we will meet an
    extension of $\piorb(\M_{12})\iso\Br_{12}(\cpone)/(\Z/2)$ 
    to which these relations lift
    but the relation $ID=1$ does not.
\end{remarks}

\begin{proof}
    \eqref{ItemAtildeRelations}
    The map $(x_0,\dots,x_{n-1})\mapsto x_0\cdots x_{n-1}$
    fibers $X^\circ(\C^*)$ over $\C^*$.
    This descends to a fibration $Y^\circ(\C^*)\to\C^*$.
    All the $\rho_j$ are paths in a single fiber, and 
    it is well-known that the fundamental group of 
    each fiber is the Artin group
    of type $\Atilde_{n-1}$, with the $\rho_j$ corresponding
    to the standard Artin generators.  See for example 
    \cite[Thm.\ 3.8]{van-der-Lek}.

    The existence of this fibration shows 
    $\Br_n(\C^*)\iso\Art(\Atilde_{n-1})\semidirect\Z$,
    where the $\Z$ is generated by any braid with total 
    winding number~$1$ around~$0$.  For example,
    the braid $t$ 
    that rotates the roots of unity one position counterclockwise;  formally:
    $x_j(t)=\zeta^j e^{2\pi i t/n}$.  Drawing a picture
    shows $t\rho_k t^{-1}=\rho_{k-1}$ for all~$k$.

    \eqref{ItemAdjoinZeroRelations}
    From the inclusion $Y^\circ(\C^*)\to Y^\circ(\C)$, one
    can work out the
    map $\Br_n(\C^*)\to\Br_n(\C)$ by using Van Kampen's theorem.
    It is surjective, with kernel normally generated by 
    the following braid:  
    every $x_k(t)$ is the constant path at $\zeta^k$ except
    $x_{n-1}(t)$, which starts at $\zeta^{n-1}$,  
    approaches~$0$, encircles it
    once negatively,
    and then returns to~$\zeta^{n-1}$.  Formally, $\Br_n(\C)$
    is
    the quotient of $\Br_n(\C^*)$ by the relation 
    $D_{n-1}t^{-1}=1$.

    Another way to say this is that $\Br_n(\C)$ is the quotient
    of $\Art(\Atilde_{n-1})$ by the relations that $D_{n-1}$
    conjugates the generators of the Artin group in the same
    way that $t$ does.  That is, by the relations
    $D_{n-1}\rho_k D_{n-1}^{-1}=\rho_{k-1}$ for all~$k$.
    Repeatedly conjugating $D_{n-1}$ by itself therefore
    yields $D_{n-2}, D_{n-3},\dots$.  On the other hand, repeatedly
    conjugating $D_{n-1}$ by itself also yields
    $D_{n-1},D_{n-1},\dots$  This establishes the relations
    in \eqref{ItemAdjoinZeroRelations}.

    To finish the proof of \eqref{ItemAdjoinZeroRelations} it is enough to show that
    adjoining the relations $D_j=D_{j'}$, for all $j,j'$,
    to $\Art(\Atilde_{n-1})$
    implies $D_{n-1}\rho_k D_{n-1}^{-1}=\rho_{k-1}$
    for all~$k$.  In fact it is enough to adjoin any single relation
    $D_j=D_{j'}$ with $j'\neq j$;
    write $D$ for the common image
    of $D_j$ and $D_{j'}$ in the quotient.
    If $j'\notin\{j\pm1\}$, then
    \eqref{EqDDecreases} establishes $D\rho_k D^{-1}=\rho_{k-1}$ for all~$k$.
    Reusing the argument of the previous paragraph,
    repeatedly conjugating $D$ by itself yields $D_{j-1},D_{j-2},\ldots=D$.
    This implies
    $D_{n-1}\rho_k D_{n-1}^{-1}=\rho_{k-1}$ for all~$k$, as desired.
    In the remaining case $j'\in\{j\pm1\}$, we may
    suppose $j'=j+1$ by  exchanging $j$ and $j'$ if necessary.
    This time \eqref{EqDDecreases}
    shows only that $D\rho_k D^{-1}=\rho_{k-1}$
    for all~$k\neq j+2$.  But $D_j=D_{j+1}$
    implies $D_j\rho_{j+2}=D_{j+1}\rho_{j+2}$, which picture-drawing
    shows is equal to $\rho_{j+1}D_j$.  
    So again we  have
    $D\rho_{j+2}D^{-1}=\rho_{j+1}$, and the 
    same argument applies.

    \eqref{ItemAdjoinInftyRelations}
    Inversion across the unit circle exchanges $0$ with $\infty$,
    preserves~$t$, 
    inverts every $\rho_k$, and exchanges $D_{n-1}$ with $I_1^{-1}$.
    Therefore we may quote the
    argument for \eqref{ItemAdjoinZeroRelations} to obtain
    two descriptions of
    $\Br_n(\C^*\cup\{\infty\})$.  First, it is the quotient of
    $\Br_n(\C^*)$ by the relation $I_1^{-1}t^{-1}=1$. Second, it is 
    the quotient of $\Art(\Atilde_{n-1})$ by the relations that
    all $I_j$ are equal.

    \eqref{ItemAdjoinBothRelations} Using
    Van Kampen's theorem twice shows that $\Br_n(\cpone)$ is the
    quotient of $\Br_n(\C^*)$ by both relations $D_{n-1}t^{-1}=1$
    and $I_1^{-1}t^{-1}=1$.  
    Therefore $\Br_n(\cpone)$ can be described as the quotient
    of $\Art(\Atilde_{n-1})$ by $I_1D_{n-1}=1$ and either the relations 
    \eqref{ItemAdjoinZeroRelations} or $\eqref{ItemAdjoinInftyRelations}$.
    If we assume \eqref{ItemAdjoinZeroRelations}, then the fact that 
    $D$ centralizes $I_1$
    forces $I_0,I_{n-1},\ldots=I_1$, establishing
    \eqref{ItemAdjoinInftyRelations}.  And vice-versa.  So the
    relation $I_1D_{n-1}=1$ can be written $ID=1$.
\end{proof}

For any elements
$g_1,\dots,g_m$ of a group, we define
\begin{equation*}
    \Delta(g_1,\dots,g_m)=(g_1 g_2 \cdots g_m)(g_1 g_2 \cdots g_{m-1})
    \cdots(g_1 g_2) g_1
\end{equation*}
In $\Br_{n+1}$, $\Delta(\rho_1,\dots,\rho_n)$ is called the 
``fundamental element''; it conjugates each $\rho_j$ to $\rho_{n+1-j}$,
and its square generates the center of $\Br_n$.

\medskip
Using the obvious
$\PGL_2\C$ action on $X(\cpone)$,
we define the moduli space $\M_{n}$ 
of $n$-point subsets of~$\cpone$
as
$$
\M_{n}=X^\circ(\cpone)\!\bigm/\!(S_n\times\PGL_2\C)
=Y^\circ(\cpone)/\PGL_2\C
$$
We assume $n\geq3$ to avoid degenerate cases.  Then
$\PGL_2(\C)$ acts freely on $X^\circ(\cpone)$, so the quotient
is a manifold.  Because $\M_{n}$ is the quotient 
of this manifold by the finite group~$S_n$, it is an orbifold.
Recall that the tuple
$T$ is our basepoint for $X^\circ(\cpone)$.
We take its images in $X^\circ(\cpone)/\PGL_2\C$
and $\M_{n}$, also
denoted~$T$,
as our basepoints when discussing their
orbifold fundamental groups.
It is well-known that  the
map $\Br_n(\cpone)\to\piorb(\M_{n})$ induced
by $Y^\circ(\cpone)\to\M_{n}$ is surjective, with
kernel is equal to the center of $\Br_n(\cpone)$, which is
isomorphic to $\Z/2$ with generator $I^n=D^{-n}$.  See
for example \cite[Thm.\ 4.5]{Birman}.  So theorem~\ref{ThmBraidGroups}
implies the following theorem.  
(The final assertion is obvious in the presence of the
fourth relation.)

\begin{theorem}
    \label{ThmPi1ModuliSpace}
    The orbifold fundamental group
    $\piorb(\M_{n},T)$ is generated by 
    $\rho_0,\discretionary{}{}{}\dots,\discretionary{}{}{}\rho_{n-1}$,
    with defining relations
    \begin{enumerate}
        \item
            \label{ItemM0nBraidRelations}
            $\rho_j\rho_k\rho_j=\rho_k\rho_j\rho_k$ or
            $\rho_j\rho_k=\rho_k\rho_j$, according to whether
            $k\in\{j\pm1\}$ or not.
        \item
            \label{ItemM0nIsAreEqualRelations}
            All the
            $I_j:=\rho_j\rho_{j+1}\cdots\rho_{j+n-1}$ coincide.
        \item
            \label{ItemM0nDsAreEqualRelations}
            All the 
            $D_j:=\rho_j\rho_{j-1}\cdots\rho_{j-n+1}$ coincide.
        \item
            \label{ItemM0nIDis1Relation}
            $ID=1$, where $I$ resp.\ $D$
            is
            the common image of the $I_j$ resp.\ $D_j$.
        \item
            \label{ItemM0nInEqualsDnEquals1}
            $I^n=D^n=1$.
    \end{enumerate}
    When $n$ is even, \eqref{ItemM0nInEqualsDnEquals1}  may be replaced
    by the relation $I^{n/2}=D^{n/2}$.\hfill\qed
\end{theorem}

\section{\texorpdfstring{Change of basepoint in $\BB^{13}$}{Change of basepoint in B13}}
\label{SecChange}

\noindent
In this section we begin working in the hyperplane arrangement
complement $\BB^{13}-\H$, with an emphasis on how the Deligne-Mostow
ball $\BDM$ lies inside~$\BB^{13}$.
We fix an $A_4$ subdiagram of
the incidence graph of $P^2\F_3$, whose $26$ nodes are the
point- and line-roots $p_j$ and $l_j$ from section~\ref{t-P2-F3-model-of-L}. 
Any two such subdiagrams are
equivalent, but for concreteness we choose 
$l_1,p_2,l_2,p_3$.  These roots are mutually orthogonal,
except that $\ip{p_2}{l_1}=\ip{p_2}{l_2}=\ip{p_3}{l_2}=\theta$.
Therefore their integral span
is a copy of
the Eisenstein lattice~$L_4$.  We call it  $L_4^{(1)}$ to distinguish  
it from two other
copies of $L_4$ introduced below.
As indicated in section~\ref{subsec-DM-lattice},
we will write $\LDM$ for the orthogonal complement
of $L_4^{(1)}$ in~$L$, and
$\BDM$ for the corresponding $9$-ball. 
See section~\ref{SecDeligneMostow} for more information about
the connection of the Deligne-Mostow ball to moduli of
$12$-tuples in~$\cpone$.
Exactly $40$ mirrors contain $\BDM$,
corresponding to the scalar classes of the $240$ roots 
of $L_4^{(1)}$.

The nodes of $P^2\F_3$, that are not joined to the $A_4$, form
an $\Atilde_{11}$ diagram, ie a $12$-gon, namely
$$
p_6,
l_{10},
p_{13},
l_4,
p_7,
l_{11},
p_{12},
l_9,
p_9,
l_8,
p_8,
l_5
$$
in cyclic order.  
We introduce alternate notation
$s_0,\dots,s_{11}$ for them, in this order.
We also write $s_{\rmA},\sB,\sC,s_{\rmD}$ for the
roots $l_1,p_2,l_2,p_3$ forming the $A_4$ diagram.
As explained in section~\ref{subsec-DM-lattice}, 
the $\E$-span of $s_0,\dots,s_{11}$ is~$\LDM$.
We write $\rho$ for the projection of $\tau$ to
$\BDM$, and  let 
$\sigma$ be any point of $\overline{\tau\rho}-\{\rho\}$.
We abbreviate the meridians 
$M_{\tau,s_0},\dots,M_{\tau,s_{11}},M_{\tau,\sA},\dots,M_{\tau,\sD}$ 
to $\tau_0,\dots,\tau_{11},\tauA,\dots,\tauD$, and similarly
with $\sigma$ in place of~$\tau$.  
(In section~\ref{SecDeligneMostow} we will extend this notation by writing
$\rho_0,\dots,\rho_{11}$ for the meridians based at~$\rho$ and
associated to $s_0,\dots,s_{11}$.  But these
will represent elements of $\piorb((\BDM-\HDM)/\PGDM,\rho)$;
because $\rho$ lies in~$\H$, it doesn't make sense to speak of
$\piorb((\BB^{13}-\H)/\PG,\rho)$.)

The following theorem is the main result
of this section. Lemma~\ref{LemMirrorsNearRho} is 
the only other result referenced later in the paper.

\begin{theorem}
    \label{ThmBasepointChange}
    The segment $\overline{\tau\rho}$ meets~$\H$ only at~$\rho$.
    For any $\sigma\in\overline{\tau\rho}-\{\rho\}$,
    the change-of-basepoint isomorphism 
    $$\piorb\bigl((\BB^{13}-\H)/\PG,\tau\bigr)\iso
    \piorb\bigl((\BB^{13}-\H)/\PG,\sigma\bigr),$$
    induced by the segment $\overline{\tau\sigma}$,
    identifies each  meridian $\tau_0,\dots,\tau_{11},\tauA,\dots,\tauD$
    based at~$\tau$ with the corresponding meridian 
    $\sigma_0,\dots,\sigma_{11},\sigmaA,\dots,\sigmaD$
    based at~$\sigma$.
\end{theorem}

\begin{proof}
    This follows immediately from the next lemma: fixing
    $j=0,\dots,11,\rmA,\dots,\discretionary{}{}{}\rmD$, and letting
    $\sigma$ vary over $\overline{\tau\rho}-\{\rho\}$, the surface
    swept out by the meridians $\sigma_j$
    misses~$\H$.
\end{proof}

\begin{lemma}
    \label{LemPolygonsMissMirrors}
    \leavevmode
    \begin{enumerate}
        \item
            \label{ItemQuadrilateral}
            Fix $j=0,\dots,11$, and let $Q$ be the totally
            real quadrilateral with vertices $\tau,\rho$, and the
            projections $\tau',\rho'$ of these points to $s_j^\perp$.
            Then $Q$ meets $s_j^\perp$ in $\overline{\tau'\rho'}$, 
            meets the
            $40$ mirrors containing $\BDM$ in
            $\overline{\rho\rho'}$, and misses all other mirrors.
        \item
            \label{ItemTriangles}
            Fix $j=\rmA,\dots,\rmD$, and let  $T$ be  the 
            totally real triangle with vertices
            $\tau$, $\rho$ and the projection $\tau'$
            of $\tau$ to $s_j^\perp$.
            Then  $T$ meets $s_j^\perp$ in $\overline{\tau'\rho}$, 
            meets the other~$39$ mirrors
            containing $\BDM$ at $\rho$ only,  and
            misses all other mirrors.
    \end{enumerate}
\end{lemma}

The stabilizer $L_3(3)\semidirect2$ of~$\tau$ contains 
an order~$24$ dihedral group $D_{24}$ 
that preserves $\{\sA^\perp, \dots, \sD^\perp\}$ and
acts faithfully and transitively on the set of
mirrors $s_0^\perp,\dots,s_{11}^\perp$. 
Half of these transformations
exchange $\sA$ with $\sD$ and
$\sB$ with $\sC$. 
Therefore it suffices to prove the $\sA,\sB$ cases
of \eqref{ItemTriangles} and the $s_0$ case of 
\eqref{ItemQuadrilateral}.   We restrict attention
to these three cases.

We will need to know the corners of the polygons 
explicitly.  All three have $\tau=(4+\sqrt3;1,\dots,1)$
as a vertex, with norm
$-6-8\sqrt3$.  
Now, the projection of $\tau$ to $L_4^{(1)}\tensor\C$
is
$$
-(3+2\sqrt3)(l_1+i p_3)-(5+3\sqrt3)(l_2+i p_2)
$$
One checks this by computing this vector's inner products with 
$l_1,p_2,l_2,p_3$
and comparing with
$\ip{\tau}{p_j}=-\theta$ and
$\ip{\tau}{l_j}=-\sqrt3$.
Subtracting it from $\tau$ gives
the other vertex that is shared by all three polygons, namely
$$
    \rho=(6\lambda;
    2\lambda,
    0, 0,
    2\lambda,
    3\lambda,
    1,1,1,1,
    2\lambda,
    3\lambda,
    1,1)
    \quad\hbox{where $\lambda=2+\sqrt3$}
$$
Computation shows  $\rho^2=\ip{\rho}{\tau}=-36-24\sqrt3$.

The remaining vertex or vertices are different in the three
cases.  First we consider the $\sA(=l_1)$ case of 
Lemma~\ref{LemPolygonsMissMirrors}\eqref{ItemTriangles}.  Then
$$
\tau'=\tau-\frac{\ip{\tau}{l_1}}{\ip{l_1}{l_1}}l_1
=\tau-\frac{-\sqrt3}{3}l_1
=\tau+l_1/\sqrt3.
$$
The calculation of $\tau'$ in the $\sB(=p_2)$ case is the same,
except that $\ip{\tau}{p_2}=-\theta$, leading to
$\tau'=\tau+i p_2/\sqrt3$. 
In both cases one computes
$$
(\tau')^2=\ip{\tau'}{\tau}=-7-8\sqrt3
\qquad\hbox{and}\qquad
\ip{\tau'}{\rho}=-36-24\sqrt3.$$
Similarly, in the quadrilateral case we have
$\tau'=\tau+i p_6/\sqrt3$, 
$$
(\tau')^2=\ip{\tau'}{\tau}=-7-8\sqrt3
\qquad\hbox{and}\qquad
\ip{\tau'}{\rho}=-37-24\sqrt3.$$
The fourth vertex of~$Q$ is $\rho'$.  Because $\rho$ differs from
$\tau$ by a vector orthogonal to $s_0(=p_6)$, we have 
$\ip{\rho}{p_6}=\ip{\tau}{p_6}=-\theta$. Therefore 
$\rho'=\rho+i p_6/\sqrt3$.  One computes
$$
(\rho')^2=
\ip{\rho'}{\rho}=
\ip{\rho'}{\tau}=
\ip{\rho'}{\tau'}=
-37-24\sqrt3.
$$
The reason that all these inner products turn out to be real is that
$ip_1,\dots,\discretionary{}{}{}ip_{13},\discretionary{}{}{}l_1,\dots,l_{13}$ have real inner products,
and all these points of $\BB^{13}$ lie in the real hyperbolic $13$-space
defined by their real span.

In each of the three cases, we have seen that the vertices of 
$T$ resp.\ $Q$ are represented by negative-norm
vectors whose real span 
is 3-dimensional and
totally real (ie contains no complex subspaces).
This justifies our claim in Lemma~\ref{LemPolygonsMissMirrors} that
$Q,T\sset\BB^{13}$ are totally real polygons.  
For use in the proof of lemma~\ref{LemPolygonsMissMirrors}, we also note
the negativity of 
the pairwise
inner products of the vectors representing the vertices.

\medskip
Our strategy for proving
lemma~\ref{LemPolygonsMissMirrors}
derives from  \cite[Appendix A]{AB-26gens}.  There,
we showed that various totally real triangles are covered by 
balls of various radii centered at the 
point 
$ p_\infty=(\theta;0,\dots,0) \in\BB^{13}$ where
all $13$ point-mirrors meet.  Then we used a computer
to enumerate all mirrors coming close enough to $p_\infty$ to meet
these balls, and
worked out how they met each triangle of interest.  We will use
the same strategy, but we will also need balls centered around
another point $c\in\BB^{13}$.  To define it, note
that there is a unique way to choose two $A_4$ diagrams in
the $\Atilde_{11}$, that are not joined to each other or to $s_0$.
Namely, we define $L_4^{(2)}$ as the span of $s_2,\dots,s_5$, and
$L_4^{(3)}$ as the span of $s_7,\dots,s_{10}$.  We define $c$
as the point of $\BB^{13}$ that is orthogonal to
$\spanof{s_0}\oplus L_4^{(1)}\oplus L_4^{(2)}\oplus L_4^{(3)}$.
Just like $p_\infty$, $c$ is represented by a norm $-3$ lattice
vector, namely
$$
c=\theta(4;1,0,0,2,2,0,0,0,0,1,2,1,0)
$$

As in \cite[Appendix A]{AB-26gens}, we call an open ball centered at
$p_\infty$ a critical ball if its boundary is tangent to some mirror.
Even though some mirrors pass through $p_\infty$, we do not count
the radius~$0$ ball as critical.  The distance from $p_\infty$ to 
the mirror of a root~$s$ is given by
$$
\sinh^{-1}\sqrt{-\frac{|\ip{p_\infty}{s}|^2}{p_\infty^2 s^2}}
=\sinh^{-1}\sqrt{\frac{|\ip{p_\infty}{s}|^2}{9}}
=\sinh^{-1}\sqrt{(0,1,3,4,7,\dots)/3}
$$
The last equality comes from $\ip{p_\infty}{s}\in\theta\E$.
The numerical values of the first few critical radii are
$$
r_1,r_2,r_3,r_4,\ldots\approx .549, .881, .987, 1.210,\dots
$$
We call $s$ a batch~$n$ root, and $s^\perp$ a batch~$n$ mirror,
around~$p_\infty$, if $s^\perp$ is tangent to the $n$th
critical ball.  We extend this language to ``batch~$0$'' in the
case $s\perp p_\infty$.
All these considerations apply verbatim with $c$ in place of $p_\infty$.

\begin{lemma}
    \label{LemPolygonsCoveredByBalls}
    The triangles in the $\sA,\sB$ cases of 
    lemma~\ref{LemPolygonsMissMirrors}\eqref{ItemTriangles}, resp.\ the quadrilateral
    in the $s_0$ case of lemma~\ref{LemPolygonsMissMirrors}\eqref{ItemQuadrilateral},
    are covered by the union of the
fourth critical balls around $p_\infty$
    and~$c$.
\end{lemma}

\begin{proof}
    In the
    quadrilateral case, we define $m=(\tau+\rho)/2\in\overline{\tau\rho}$ and
    $m'=(\tau'+\rho')/2\in\overline{\tau'\rho'}$.  
    Because balls and (totally real) polygons are convex, it
    is enough to check that
    $\tau,\tau',m,m'$ lie at distance${}<r_4=\sinh^{-1}\sqrt{7/3}\approx1.210$
    from $p_\infty$, and that $m,m',\rho,\rho'$ lie at distance${}<r_4$
    from~$c$.  In the triangle cases we use the same argument,
    except that we define $m'=(\tau'+\rho)/2$, and we replace
    the quadrilateral $m,m',\rho,\rho'$ by the triangle
    $m,m',\rho$.  Here are the data shared by all 3 cases:
    \begin{align*}
        d(p_\infty,\tau)&{}=\cosh^{-1}\sqrt{\textstyle\frac{1}{2}+\frac{2}{3}\sqrt3}
        \approx .740
        \\
        d(p_\infty,m)&{}=\cosh^{-1}\sqrt{\textstyle\frac{1303}{1034}+\frac{1676}{1551}\sqrt3}
        \approx 1.172
        \\
        d(c,m)&{}=\cosh^{-1}\sqrt{\textstyle\frac{794}{517}+\frac{1376}{1551}\sqrt3}
        \approx 1.161
        \\
        d(c,\rho)&{}=\cosh^{-1}\sqrt{\textstyle\frac{5}{4}+\frac{13}{18}\sqrt3}
        \approx 1.032
    \end{align*}
    In the triangle resp. quadrilateral cases we have 
    \begin{equation*}
        d(c,m')=\cosh^{-1}\sqrt{\textstyle\frac{1828}{1195}+\frac{1056}{1195}\sqrt3}
        \approx 1.158
        \hbox{ resp.\ }
        \cosh^{-1}\sqrt{\textstyle\frac{1994}{1319}+\frac{1152}{1319}\sqrt3}
        \approx 1.151.
    \end{equation*}
    In the $\sA$ case resp.\ the other two cases we have
    \begin{equation*}
        d(p_\infty,\tau')=\cosh^{-1}\sqrt{\textstyle\frac{320}{429}+\frac{96}{143}\sqrt3}
        \approx .848
        \hbox{ resp.\ }
        \cosh^{-1}\sqrt{\textstyle\frac{59}{143}+\frac{96}{143}\sqrt3}
        \approx .700.
    \end{equation*}
    In the $\sA$, $\sB$ and quadrilateral cases respectively we have
    \begin{align*}
        d(p_\infty,m')=
        {}&
        \cosh^{-1}\sqrt{\textstyle\frac{4996}{3585}+\frac{256}{239}\sqrt3}
        \approx 1.195
        ,\  
        \cosh^{-1}\sqrt{\textstyle\frac{1483}{1195}+\frac{1296}{1195}\sqrt3}
        \approx 1.170,
        \\
        &\hbox{or }
        \cosh^{-1}\sqrt{\textstyle\frac{3103}{2638}+\frac{1452}{1319}\sqrt3}
        \approx 1.163.
    \end{align*}
    Finally, in the quadrilateral case we have
    $
    d(c,\rho')=\cosh^{-1}\sqrt{\textstyle\frac{443}{359} + \frac{256}{359}\sqrt3}
    \approx 1.024.
    $
    One may
    use the identity $\cosh^2=\sinh^2+1$    to avoid approximation
    when checking that these distances are less than $\sinh^{-1}\sqrt{7/3}$.
    This reduces one to checking
    that each radicand is less than $10/3$.
\end{proof}

\begin{proof}[Proof of lemma~\ref{LemPolygonsMissMirrors}.]
    As mentioned above, it suffices to prove this for the  quadrilateral
    and two
    triangles treated in lemma~\ref{LemPolygonsCoveredByBalls}.
    By that lemma, the only mirrors that can meet these polygons 
    are the $0$th, $1$st, $2$nd and $3$rd shell mirrors around
    $p_\infty$ and~$c$.  In \cite[Lemma~A.12]{AB-26gens} we explained how to enumerate these
    mirrors around~$p_\infty$, and below we will explain the corresponding 
    enumeration around~$c$.  
    Now suppose given a triangle $T\sset\BB^{13}$
    whose vertices are represented by vectors whose norms and inner products are
    negative.  Following \cite[Lemma~A.1]{AB-26gens},
    $T$ meets a mirror $s^\perp$ if and only if
    the origin lies in the triangle $\ip{T}{s}\sset\C$ whose vertices are
    the inner products of these vectors with~$s$.  In this way we 
    found all the mirrors meeting the polygons.
    We carried out all calculations using exact
    arithmetic in the field $\Q(\theta,\sqrt3)$.
    The results are as stated in the lemma.
\end{proof}

It remains to explain the enumeration of mirrors near~$c$.  We work with respect
to the basis $c,s_0;\sA,\sB,\sC,\sD
;s_5,s_4,s_3,s_2;s_7,s_8,s_9,s_{10}$ for $\C^{13,1}$,
whose inner product matrix is
$$
\begin{pmatrix}
    -3&0\\
    0&3
\end{pmatrix}
\oplus
\begin{pmatrix}
    3&\thetabar&0&0\\
    \theta&3&\theta&0\\
    0&\thetabar&3&\thetabar\\
    0&0&\theta&3
\end{pmatrix}
\oplus
\begin{pmatrix}
    3&\thetabar&0&0\\
    \theta&3&\theta&0\\
    0&\thetabar&3&\thetabar\\
    0&0&\theta&3
\end{pmatrix}
\oplus
\begin{pmatrix}
    3&\thetabar&0&0\\
    \theta&3&\theta&0\\
    0&\thetabar&3&\thetabar\\
    0&0&\theta&3
\end{pmatrix}
$$
(The backwards ordering of the $4$-tuple $s_5,\dots,s_2$ makes its inner product
matrix coincide with those of the other $4$-tuples.)
To
avoid confusion, we will 
use square brackets when writing components of vectors with
respect to this basis.


\begin{lemma}
$L$ consists of all vectors $[\frac{a}{\theta},\frac{b}{\theta};\vec{v}_1;\vec{v}_2;\vec{v}_3]$
where $a,b\in\E$ are congruent mod~$\theta$ and 
    $\vec{v}_1,\vec{v}_2,\vec{v}_3\in L_4$.
\end{lemma}

\begin{proof}
    This is just a computation, but we indicate why it should be true.
The $4$-tuple in each of the last three blocks of basis vectors
generates a copy of $L_4$.
Recall that $L_4$ is a summand, of any lattice
that contains it and has all inner products divisible by~$\theta$. Therefore the
semicolons delimit summands of~$L$.  
    The only possibility for the remaining summand is
    $\twobytwo{0}{\thetabar}{\theta}{0}$.  Therein, $c$ is 
    a primitive lattice vector in the orthogonal complement
    of the norm~$-3$ vector~$s_0$, and therefore must have norm~$3$ (making
    visible a fact we 
    already used).  There are only two proper
    enlargements of the lattice $\twobytwo{-3}{0}{0}{3}$
    to a lattice with all inner products divisible by~$\theta$, got by adjoining
    $(s_0\pm c)/\theta$.
    We chose the sign of~$c$ so that
    $(s_0+c)/\theta\in L$.
\end{proof}

\begin{lemma}
    \label{LemBatchesAroundC}
    The roots in batches $0,\dots,3$ around~$c$ appear in table~\ref{TabBatchesAroundC}.
\end{lemma}

\begin{table}
    \begin{tabular}{cccccr}
        \toprule
        batch&$a$&$b$&$|b|^2$&norms of $\vec{v}_1,\vec{v}_2,\vec{v}_3$&mirrors
        \\
        \midrule
        $0$&$0$&$0$&$0$&$3,0,0$&$120$
        \\
        &$0$&$\pm\omega^j\theta$&$3$&$0,0,0$&$1$
        \\
        \midrule
        $1$&$1$&$\omega^j$&$1$&$3,0,0$&$2160$
        \\
        &$1$&$-2\omega^j$&$4$&$0,0,0$&$3$
        \\
        \midrule
        $2$&$\theta$&$0$&$0$&$6,0,0$&$6480$
        \\
        &$\theta$&$0$&$0$&$3,3,0$&$172800$
        \\ 
        &$\theta$&$\pm\omega^j\theta$&$3$&$3,0,0$&$4320$
        \\
        \midrule
        $3$&$-2$&$\omega^j$&$0$&$6,0,0$&$6480$
        \\
        &$-2$&$\omega^j$&$0$&$3,3,0$&$518400$
        \\
        &$-2$&$-2\omega^j$&$4$&$3,0,0$&$2160$
        \\
        &$-2$&$\omega^j(3+(\omega\hbox{ or }\omegabar))$&$7$&$0,0,0$&$6$
        \\
        \bottomrule
    \end{tabular}
    \bigskip
    \caption{The batch $0,\dots,3$ mirrors around $c$; see lemma~\ref{LemBatchesAroundC}.
    We list all batch~$0$ roots, and one root from each scalar class of
    roots in batches $1$, $2$ and~$3$.  Listed roots have the form
    $[\frac{a}{\theta},\frac{b}{\theta};\vec{v}_1;\vec{v}_2;\vec{v}_3]$
    where $a,b$ appear in the table, and the norms of
    $\vec{v}_1,\vec{v}_3,\vec{v}_3\in L_4$ are as specified, up to permutation.
    When present, $j$ varies over $\{0,1,2\}$.  The last column gives
    the number of mirrors arising from the listed roots.}
    \label{TabBatchesAroundC}
\end{table}

\begin{proof}
    Suppose $s=[\frac{a}{\theta},\frac{b}{\theta};\vec{v}_1;\vec{v}_2;\vec{v}_3]$ 
    is a root in batch $0$, $1$, $2$ resp.\ $3$.  Then 
    $|\ip{s}{c}|^2=3\cdot(\hbox{$0$, $1$, $3$ resp.\ $4$})$.
    Since $\ip{s}{c}=-3\frac{a}{\theta}$, this shows $|a|^2=0$, $1$, $3$, resp.\ $4$.
    We may
    suppose $a=0$, $1$, $\theta$ resp.\ $-2$ after scaling by a unit.
    Since $s^2=3$, we have 
    \begin{equation}
        \label{EqFoo}
        \textstyle
        3|\frac{b}{\theta}|^2+\vec{v}_1^{\,2}+\vec{v}_2^{\,2}+\vec{v}_3^{\,2}
        =
        3+3|\frac{a}{\theta}|^2
        =
        \hbox{$3$, $4$, $6$ resp.\ $7$}
    \end{equation}
    In particular, $|b|^2$ is at most the right side.  
    After imposing the condition 
    $$\hbox{$b\cong a\cong0$, $1$, $0$ resp.\ $1$\quad (mod~$\theta$),}$$
    $b$ must be one of the possibilities listed in the table.  From
    \eqref{EqFoo} follows $\vec{v}_1^{\,2}+\vec{v}_2^{\,2}+\vec{v}_3^{\,2}=3+|a|^2-|b|^2$.
    Since the vectors of $L_4$ have norms $0,3,6,9,\dots$, the norms
    of $\vec{v}_1,\vec{v}_2,\vec{v}_3$ are as stated in table~\ref{TabBatchesAroundC},
    up to permutation.
    Counting the mirrors of each type uses
    the fact that  $L_4$ has $240$ vectors of norm~$3$
    and $2160$ of norm~$6$.
\end{proof}

To carry out the proof of lemma~\ref{LemPolygonsMissMirrors}, we prepared lists of the norm~$3$ and
$6$ vectors in $L_4$, and used these to make lists
of all possibilities for $(\vec{v}_1,\vec{v}_2,\vec{v}_3)$, 
which in turn we used to construct the roots in table~\ref{TabBatchesAroundC}.
We converted these roots to our usual coordinates, 
and then computed their mirrors' intersections with 
the polygons.  
The enumeration of mirrors 
near~$c$ also allows us work out $\rho$'s nearest mirrors:

\begin{lemma}
    \label{LemMirrorsNearRho}
            The components of~$\H$ that come 
            nearest~$\rho$,
            other than those that pass through it,
            are 
            $s_0^\perp,\dots,s_{11}^\perp$.  
\end{lemma}

    \begin{proof}
        One can check
        \begin{equation*}
            \textstyle
        \sinh^{-1}\sqrt{-\frac{1}{12}
    +\frac{1}{18}\sqrt3}
    +\cosh^{-1}\sqrt{\frac{5}{4}+
    \frac{13}{18}\sqrt3}
        <
            \sinh^{-1}\sqrt{\frac{7}{3}}
        \end{equation*}
    The 
        terms on the left are the distances from $\rho$
        to 
        $s_j^\perp$ and~$c$ respectively.
        They are approximately 
$.113$ and $1.032$.  The right side is $r_4\approx1.210$.
        It follows that the mirrors nearest~$\rho$,
        subject to missing it, lie in 
        batches $0,\dots,3$
    around~$c$.  Examining these batches proves the lemma.
        To 
        avoid approximation when checking the inequality,
        apply $\sinh$ to both sides and then use its
        ``angle sum'' formula.
\end{proof}

\section{The Deligne-Mostow 9-ball}
\label{SecDeligneMostow}

\noindent
In their celebrated papers \cite{Deligne-Mostow} and 
\cite{Mostow}, 
Deligne and Mostow related many ball quotients to various moduli
spaces of tuples in $\cpone$, and their work 
has been re-examined from several points of view
\cite{Thurston}\cite{CHL}.
The largest-dimensional example  is an isomorphism between
$(\BDM-\HDM)/\PGDM$ and the moduli space $\M_{12}$
of unordered $12$-tuples of distinct points in $\cpone$.
Here $\BDM=\BB(\LDM)$, where the lattice $\LDM$
was defined in section \ref{subsec-DM-lattice}, 
$\HDM$ is the union of the mirrors of the
roots of~$\LDM$, and $\GDM$ is the isometry group of~$\LDM$.
The Deligne-Mostow isomorphism is not very explicit.  Our
goal is to make it more explicit, by
saying which point of~$\BDM/\PGDM$ corresponds
to the standard basepoint for $\M_{12}$, and which loops
in $(\BDM-\HDM)/\PGDM$ correspond to the standard generators of 
$\Br_{12}(\cpone)$.
This result is interesting without any monstrous connection at all.

Inside~$L$, the sublattice $\LDM$ is the orthogonal 
complement of the span $L_4$
of specific roots $\sA,\dots,\sD$ of~$L$.  
Its isometry group $\GDM$ 
is a subgroup of~$\Gamma$ because $\LDM$ is a summand of~$L$.
Although we defined the hyperplane arrangement $\HDM$ in terms of the
roots of~$\LDM$, it can also be defined by as the set of hyperplanes
arising from intersections of~$\BDM$ with components of~$\H$.
(See lemma~\ref{LemMirrorsMeetingBDM} below.)
In section~\ref{SecChange}
we defined $\rho$ as the projection of $\tau\in\BB^{13}$
to $\BDM$, and we defined specific roots 
$s_0,\dots,s_{11}$ of $\LDM$.  By lemma~\ref{LemMirrorsNearRho}, their mirrors
are the mirrors of $\HDM$ that come closest to~$\rho$.  We define
$$
\rho_j:=\hbox{ the meridian }M_{\rho,s_j}\in\piorb((\BDM-\HDM)/\PGDM,\rho).
$$
Now we can state the main theorem of this section.

\begin{theorem}
    \label{ThmDeligneMostow}
    There is a complex-analytic orbifold isomorphism between
  $\M_{12}$   and $(\BDM-\HDM)/\PGDM$, which identifies
    their respective basepoints $T$ and $\rho$, and identifies
    the $\rho_j\in\piorb\bigl((\BDM-\HDM)/\PGDM,\rho)$ with
    the elements of $\piorb(\M_{12},T)$ denoted by the same
    symbols in section~\ref{SecBraidGroups}.  

    In particular, 
    $\rho_0,\ldots,\rho_{11}$
    generate $\piorb((\BDM-\HDM)/\PGDM,\rho)$, with
    defining relations stated in the $n=12$ case of
    theorem~\ref{ThmPi1ModuliSpace}.
\end{theorem}

    Deligne and Mostow constructed this isomorphism, but made no
    statements about basepoints and fundamental group elements.
    Our job is to prove the extra statements.
    First we recall 
    the notation of section~\ref{SecBraidGroups} in the case
    $\Sigma=\cpone$ and $n=12$.   
    Namely,
    $X$ and $Y$ are the spaces of
    ordered and unordered $12$-tuples in $\cpone$,  $X^\circ$ and 
    $Y^\circ$
    are their subspaces consisting of 
    $12$-tuples of distinct points, and $\M_{12}=Y^\circ/\PGL_2\C$.
    Additionally, we define $\Xst$
    as the subset of~$X=(\cpone)^{12}$ 
    consisting of ordered $12$-tuples in~$\cpone$
    with no
    points of multiplicity${}>5$, and $\Yst=\Xst/S_{12}$.
    It is not important in this paper, but 
    the superscript indicates stability in the sense of geometric
    invariant theory.

    The essential (for us)
    properties of the Deligne-Mostow isomorphism are:
    \begin{enumerate}
        \item
            \label{ItemHomeomorphism}
            There is a complex algebraic variety isomorphism 
            $f$ from
            $\Yst/\PGL_2\C$ to $\BDM/\PGDM$.
            (All we need is 
            a homeomorphism.)
        \item
            \label{ItemOrbifoldIso}
            The restriction of $f$ to $\M_{12}=Y^\circ/\PGL_2\C$
            is a complex analytic
            orbifold isomorphism, onto $(\BDM-\HDM)/\PGDM$.
        \item
            \label{ItemDiscriminant}
            Those points of $\Yst/\PGL_2\C$
            that are represented by $12$-tuples with 
            a single multiple point, that point having multiplicity~$2$,
            are identified under~$f$ with the points of $\BDM$
            that lie in exactly one component of~$\HDM$ (or rather,
            with their images in
            $\BDM/\PGDM$).
        \item
            \label{ItemUniversalTriflection}
            $X^\circ/\PGL_2\C$ resp.\ 
            $\BDM-\HDM$ is the covering space of $\M_{12}\iso(\BDM-\HDM)/\PGDM$ 
            corresponding to the subgroup 
            of its orbifold fundamental group
            that is generated by the squares resp.\ cubes
            of all meridians.  (Meridians in $\piorb((\BDM-\HDM)/\PGDM)$
            correspond to  the conjugacy class in $\piorb(\M_{12})$
            that contains the standard
            braid generators.)
    \end{enumerate}

    \noindent
We prove our earlier claim about the two descriptions
of~$\HDM$, and then proceed to the proof of theorem~\ref{ThmDeligneMostow}.

\begin{lemma}
    \label{LemMirrorsMeetingBDM}
    Every mirror of~$L$ that meets $\BDM$
    is the mirror of a root of~$\LDM$, except for the
    $40$ mirrors which contain~$\BDM$.
\end{lemma}

\begin{proof}
    Suppose $\BDM$ meets the mirror of a root~$r$.
    Then the $\E$-span of $r$ and $L_4$ is positive-definite.
    Also, all its inner products are divisible by~$\theta$.
    Therefore $L_4$ is a summand of it. 
    Since $r^2$ and the
    minimal norm of $L_4$ are both~$3$, we must have
    $r\in L_4$ or $r\perp L_4$.
\end{proof}

    \begin{lemma}
        \label{LemTandRHOcorrespond}
        The isomorphism $f$ identifies (the image of) $T$ in $\M_{12}$,
        with (the image of) $\rho$ in $(\BDM-\HDM)/\PGDM$.
    \end{lemma}

    \begin{proof}
        The main point is that $\rho$'s $\PGDM$-stabilizer is~$D_{24}$.
        This uses lemma~\ref{LemMirrorsNearRho}, ie that
        the mirrors of $s_0,\dots,s_{11}$ are the components of~$\HDM$
        closest to~$\rho$.  So any element of the stabilizer permutes them,
        and indeed
        is determined by how it permutes them. (Consider its action
        on their points closest to~$\rho$.)   Since these roots
        form a $12$-gon, the stabilizer is no larger than~$D_{24}$.
        And a $D_{24}$ in the stabilizer is 
        visible (see section~\ref{subsec-DM-lattice}).

        By the orbifold isomorphism, the corresponding point of $X_{0,12}/\PGL_2\C$
        also has stabilizer~$D_{24}$, this time in the  $S_{12}$-action
        on $X^\circ/\PGL_2\C$.  Up to projective equivalence there is only one 
        stable (unordered)
        $12$-tuple with stabilizer isomorphic to~$D_{24}$, namely~$T$.
    \end{proof}

    Let $\rho'$ be the projection of $\rho$ to $s_0^\perp$.  
    The segment $\overline{\rho\rho'}$ corresponds via~$f$ to some
    motion in~$\cpone$ of the $12$ points of~$T$.  
 The rest of the proof amounts to formulating this precisely and 
 then identifying the motion.
 To start, observe that the orbifold covering spaces
    $X^\circ/\PGL_2\C\to \M_{12}$ and $\BDM-\HDM\to(\BDM-\HDM)/\PGDM$
    are manifold covers of the same orbifold.  
    Therefore \eqref{ItemOrbifoldIso} implies:
    there is a diffeomorphism $\ftilde$ from 
    a neighborhood $E_T$ of (the image of) $T$ in $X^\circ/\PGL_2\C$,
    to a neighborhood $E_\rho$ of $\rho$ in $\BDM-\HDM$, that lies over~$f$
    and is $D_{24}$-equivariant.

    Together with analytic continuation, this allows us to transfer
    any path in~$\BDM-\HDM$ beginning at~$\rho$,
    to a path in $X^\circ/\PGL_2\C$ beginning at~$T$.   
    In particular:
    there exists 
    a unique path $\gamma:[0,1)\to X^\circ/\PGL_2\C$, 
    whose initial segment corresponds
    to $\overline{\rho\rho'}$ under~$\ftilde$, and which 
    has the same projection as $\overline{\rho\rho'}-\{\rho'\}$
    to $\M_{12}\iso(\BDM-\HDM)/\PGDM$.
    Because $\Xst/\PGL_2\C$ is finite over $\Yst/\PGL_2\C$, 
    $\gamma$ extends continuously to $[0,1]$.  Finally, because
    $\Xst$ is a $\PGL_2\C$-bundle over $\Xst/\PGL_2\C$,
    we may lift $\gamma$ to
    a path $\beta(t)=(\beta_0,\dots,\beta_{11})\in\Xst$ of
    ordered $12$-tuples in~$\cpone$.  

    Our analysis of $\beta$ will rely on
    a study of certain
    anti-holomorphic
    involutions of $\M_{12}$ and $\BDM$.
    Write $\chi:\cpone\to\cpone$ for the inversion map 
    across the unit circle, ie $z\mapsto 1/\zbar$.
    We  also write $\chi$ for the induced maps on
    $\Xst$ and $\Yst$ and their quotients by $\PGL_2\C$.  
    As a self-map  of $\cpone$, the anti-holomorphic involution
    $\chi$ preserves~$T$ pointwise
    and commutes with $T$'s setwise stabilizer 
    $D_{24}\sset\PGL_2\C$.  Another way to say this is that
    $\chi$'s action on $X^\circ/\PGL_2\C$ preserves the image 
    therein of $T$,
    and commutes with the $D_{24}\sset S_{12}$ stabilizing this
    image.  
    
    We now construct an antiholomorphic involution $\chihat$ of
    $\BDM$ that fixes~$\rho$ and 
    corresponds to $\chi$ under $\ftilde$.  The action of~$\chi$ on 
    $\M_{12}=(\BDM-\HDM)/\PGDM$
    sends meridians to inverses of meridians, and therefore preserves
    the subgroups of its orbifold fundamental group that correspond
    to the covers $X^\circ/\PGL_2\C$ and $\BDM-\HDM$.  So $\chi$
    has a lift to an antiholomorphic involution of $\BDM-\HDM$
    that fixes~$\rho$.  In fact, it has~$24$ such lifts, differing
    by composition with elements of $D_{24}$.
    Since $\ftilde:E_T\to E_\rho$ is an isomorphism of
    orbifold covering spaces, it identifies the
    lifts of $\chi$ to $E_T$ with the lifts of $\chi$ to $E_\rho$.
    We define $\chihat$ to be the lift to
    $\BDM-\HDM$ that corresponds to~$\chi$ under~$\ftilde$.
    (We can now discard $E_T$ and $E_\rho$.)
    Then the Riemann extension theorem (applied to
    the composition of $\chihat$ with any chosen anti-holomorphic
    involution of the
    ball) shows that
    $\chihat$ extends to  $\BDM$.  We remark that anti-holomorphic
    automorphisms of the complex ball are complex hyperbolic
    isometries.

    \begin{lemma}
        \label{Lem2possibilitiesForChihat}
        There are only two possibilities for~$\chihat$: it is either
        the anti-holomorphic involution $\chihat_0$  
        of $\BDM$ got from complex conjugation in the $P^2\F_3$
        model of~$L$, or the composition of $\chihat_0$ with 
        the central involution  $Z$ 
        of the $\PGDM$-stabilizer $D_{24}$ of~$\rho$.
    \end{lemma}

    \begin{proof}
    Since $\chihat$
    preserves $\rho$ and $\HDM$, it permutes the mirrors nearest~$\rho$,
        which are
    $s_i^\perp$.  Since it commutes with~$D_{24}$, it either
    preserves each of them, or else sends each $s_i^\perp$ to
    $s_{i+6}^\perp$.  In either case, this
    determines the action of $\chihat$ on
    the points of the $s_i^\perp$ that are nearest~$\rho$, and
    hence the action on all of~$\BDM$.  The first of these possibilities
    is realized by  $\chihat_0$, and   the second by
    $\chihat_0\circ Z$.
    \end{proof}

    We now return to our 
    path~$\beta(t)=(\beta_0(t),\dots,\beta_{11}(t))\in\Xst$.  
    The key ingredient in the
    analysis is that $\overline{\rho\rho'}$ is $\chihat_0$-invariant.
    Also important is that the $\beta_i(t)$ are all distinct, for each $t$
    except  $t=1$. Then, two of the $\beta_i(t)$ become equal,
    while the rest remain distinct from them and from each other.
    This follows from the fact that $\rho'$ lies in just one component
    of~$\HDM$.  One might skip the next proof for now, since 
    the lemma after it has a similar but simpler proof.

    \begin{lemma}
        \label{LemExcludeBadChihatCandidate}
        $\chihat=\chihat_0$.
    \end{lemma}

    \begin{proof}
        Suppose to the contrary, so $\chihat=\chihat_0\circ Z$.  
        Because $\overline{\rho\rho'}$ is $\chihat_0$-invariant,
$\gamma$ 
    is $(\chi\circ Z)$-invariant,
    where $Z$ is still the central involution of~$D_{24}$, but now
    acting on $\Xst/\PGL_2\C$.  That is,
    there exists $g_t:[0,1]\to\PGL_2\C$ such that 
    $(\chi\circ Z)(\beta(t))=g_t(\beta(t))$.  
    Since $Z$ acts on $X=(\cpone)^{12}$ by cyclically permuting
    the factors by six positions,
    we have
    \begin{equation}
        \label{EqAntiInvolution2}
    \bigl(\chi(\beta_6(t)),\dots,\chi(\beta_{11}(t)),\chi(\beta_0(t)),\dots,
        \chi(\beta_5(t)\bigr)=
        \bigl(g_t(\beta_0(t)),\dots,g_t(\beta_{11}(t))\bigr)
    \end{equation}
    Choose $i,j$ so that $\beta_i,\beta_j,\beta_{i+6},\beta_{j+6}$
    are not involved in the collision as $t\to1$.  From 
    \eqref{EqAntiInvolution2} this $4$-tuple is projectively
    equivalent to 
    $\chi(\beta_{i+6}),\chi(\beta_{j+6}),\chi(\beta_i),
    \discretionary{}{}{}\chi(\beta_j)$.  The latter is projectively
    equivalent to $\betabar_{i+6},\betabar_{j+6},\betabar_i,\betabar_j$,
    because $\chi$ and the usual complex conjugation on $\cpone$ differ
    by a projective transformation.  This gives an  equality of
    cross-ratios:
    $$
    \frac{(\beta_i-\beta_j)(\beta_{i+6}-\beta_{j+6})}{(\beta_i-\beta_{i+6})(\beta_j-\beta_{j+6})} =
    \frac{(\betabar_{i+6}-\betabar_{j+6})(\betabar_i-\betabar_j)}{(\betabar_{i+6}-\betabar_i)(\betabar_{j+6}-\betabar_j)}
    $$
    In particular, $\beta_i,\beta_j,\beta_{i+6},\beta_{j+6}$ have real
    cross-ratio. 

    A $4$-tuple has real cross ratio if and only if it can be carried into
        the unit circle by a projective transformation.  So,
        by applying a time-dependent projective transformation,
        we may suppose that $\beta_i(t),\beta_j(t),\beta_{i+6}(t),\beta_{j+6}(t)$
        lie on the unit circle, for all~$t$.   At $t=0$,
        $\{\beta_i,\beta_{i+6}\}$ separates $\beta_j$ and
        $\beta_{j+6}$ from each other on the circle.  
        These four points remain distinct during the
        motion, so the same holds
        for all~$t$.  Therefore the
        geodesics $\overline{\beta_i(t)\beta_{i+6}(t)}$ and
        $\overline{\beta_j(t)\beta_{j+6}(t)}$, in the Poincar\'e disk
        bounded by the unit circle, meet each other.  By applying
        a time-dependent projective transformation, preserving the unit circle,
        we may suppose without loss that the intersection point is the origin.
        In particular, $\beta_i(t)$ and $\beta_{i+6}(t)$ are antipodal
        for all~$t$, and similarly with $j$ in place of~$i$.

        We revisit \eqref{EqAntiInvolution2}
    in light of this information.  It says that $g_t$ acts by negation
    on four points of the unit circle.  This uniquely determines~$g_t$
    as $z\mapsto-z$, for all~$t$.  
        It follows that $\beta_{k+6}(t)=-1/\betabar_k(t)$ for all $k$.
        In particular, the unordered $12$-tuple underlying~$\beta(t)$
    is preserved by the fixed-point-free
    selfmap $z\mapsto-1/\bar z$ of $\cpone$.
    This is a contradiction: $\beta(1)$ has a single collision point, hence
    admits no such symmetry.
    \end{proof}

    \begin{lemma}
        \label{LemBraidRepresentingSegment}
    The path $\gamma$ in $\M_{12}$, corresponding to $\overline{\rho\rho'}$,
        is represented by a path 
        $(\beta_0(t),\dots,\beta_{11}(t))\in (\cpone)^{12}$, satisfying
        \begin{enumerate}
            \item
                $\beta_i(t)$ lies in the unit circle, for all $i$ and $t$;
            \item
                for each $t\in[0,1)$, the $\beta_i(t)$ are distinct;
            \item
                two of the $\beta_i(1)$ coincide and the rest are 
                distinct from them and each other.
        \end{enumerate}
    \end{lemma}

    \begin{proof}
        By the previous two lemmas, $\chihat=\chihat_0$.  So
    the 
    $\chihat_0$-invariance of~$\overline{\rho\rho'}$ implies the
    $\chi$-invariance of $\gamma$.
    Ie, $\chi(\beta(t))$
    is projectively equivalent to $\beta(t)$ for all~$t$, which is to 
    say that there 
    exists $g_t:[0,1]\to\PGL_2\C$ such that
    \begin{equation}
        \label{EqAntiInvolution}
        \bigl(\chi(\beta_0(t)),\dots,\chi(\beta_{11}(t))\bigr)=
        \bigl(g_t(\beta_0(t)),\dots,g_t(\beta_{11}(t))\bigr).
    \end{equation}
    for all~$t$.
    We choose $i,j,k$ not involved in the collision; in particular, 
    $\beta_i(t),\beta_j(t),\discretionary{}{}{}\beta_k(t)$ remain
    distinct as $t\to1$.  
        Because $\PGL_2\C$ acts $3$-transitively on~$\cpone$,
        we may apply a time-dependent automorphism
    of $\cpone$, to suppose that 
    $\beta_i(t),\beta_j(t),\beta_k(t)$ 
    lie on the unit circle for all~$t$.
    That is, they are $\chi$-invariant, hence by \eqref{EqAntiInvolution}
    also $g_t$-invariant.  Only the trivial element of~$\PGL_2\C$
    fixes three points, so $g_t$ is the identity for all~$t$.
    Now \eqref{EqAntiInvolution} says that 
    $\beta_i(t)$ is fixed by~$\chi$, for all $i$ and~$t$.
    We have exhibited a representative
    $\beta:[0,1]\to\Xst$ for $\overline{\rho\rho'}$, consisting of
    $12$ points moving on the unit circle, that remain distinct until
    $t=1$, when two neighbors collide but no other collisions take place.
        Here we are using property \eqref{ItemDiscriminant} of the Deligne-Mostow
        isomorphism.
    \end{proof}

\begin{proof}[Proof of theorem~\ref{ThmDeligneMostow}]
    Our goal is to 
    understand the motions of $12$-tuples in~$\cpone$, corresponding
    to the meridians $\rho_0,\dots,\rho_{11}\in\piorb((\BDM-\HDM)/\PGDM,\rho)$.
    By definition, $\rho_0$ 
    follows $\overline{\rho\rho'}$ until very near $\rho'$, then 
    travels one third of the way around $s_0^\perp$, ending on 
    $S_0(\overline{\rho'\rho})$, and then continues along 
    $S_0(\overline{\rho'\rho})$ until arriving at $S_0(\rho)$.

    We already transferred $\overline{\rho\rho'}$ to a path~$\gamma$
    in $\Xst/\PGL_2\C$.  And $\gamma$ is  represented by a path $\beta$ 
    of ordered $12$-tuples in~$\cpone$, that remain
    distinct until $t=1$, when two neighbors collide,
    say 
    $\beta_0(t),\beta_1(t)$.
    By this choice of labeling,
    $\gamma(1)$ lies in the component $C_{01}$ of $(\Xst-X^\circ)/\PGL_2\C$
    fixed by the involution $A=(01)\in S_{12}$.  
    Because there is only one collision among $\beta_0(1),\dots,\beta_{11}(1)$,
    this is the only component of $(\Xst-X^\circ)/\PGL_2\C$ 
    containing $\gamma(1)$, just as $s_0^\perp$ is the only
    component of~$\H$ containing $\rho'$.
    Because $\Xst\to\Yst$
    has order~$2$ branching along~$C_{01}$, rather than order~$3$,
    we can now identify
    the path in~$X^\circ/\PGL_2\C$ corresponding to~$\rho_0$.  It
    begins at (the $\PGL_2\C$-orbit of) $T$,
    travels along $\gamma(t)$ until $t$ is very near~$1$ (say $t=t_0$), so
    that $\gamma(t_0)$ is very near~$C_{01}$.  Then it travels
    halfway around~$C_{01}$, ending at $A(\gamma(t_0))$,
    and then continues along $A(\reverse(\gamma))$ until arriving at
    $A(T)$.  

    This path is easy to lift to $X^\circ$, because we have
    already lifted $\gamma$ to $\beta$.
    It begins at~$T$ and 
    follows $\beta(t)$ until $t=t_0$, so that  $\beta_0,\beta_1$
    have almost
    collided.  Then $\beta_0,\beta_1$ move positively around each other,
    so that they swap places, ending at the $A(\beta(t_0))$.
    Then it follows $A(\reverse(\beta))$ all the way to
    $A(T)$.  This is one of the standard generators given in section~\ref{SecBraidGroups}
    for 
    the braid group $\Br_{12}(\cpone)$.   

    We have proven that
    the isomorphism $f:\M_{12}\to(\BDM-\HDM)/\PGDM$
    sends the standard generators
    for $\Br_{12}(\cpone)$, called 
    $\rho_i$ in section~\ref{SecBraidGroups}, to the meridians in
    $\piorb((\BDM-\HDM)/\PGDM,\rho)$ called $\rho_i$ in this section,
    up to labeling.  Considering which pairs of
    generators braid, and which pairs commute, shows that the labeling
    may be chosen to be $\rho_i\leftrightarrow\rho_i$.
\end{proof}

\section{\texorpdfstring{A neighborhood of $\BDM$, modulo its stabilizer}{A
neighborhood of the Deligne-Mostow ball, modulo its stabilizer}}
\label{SecNeighborhood}

\noindent
Recall from  section~\ref{SecChange}
that $\BDM$ means the $9$-ball orthogonal to the roots
$\sA,\dots,\sD$
that form our chosen
$A_4$ subdiagram of $P^2\F_3$.  
Our goal in this section is to define a suitable neighborhood~$U$
of $\BDM$, invariant under the setwise stabilizer $\PGDMsw$
of $\BDM$, and write down a presentation of
$\piorb\bigl((U-\H)/\PGDMsw\bigr)$.  
Theorem~\ref{ThmDeligneMostow} leads to
a surjection from this group 
to the orbifold
fundamental
group of the moduli space $\M_{12}$ of 
$12$-point subsets of $\C P^1$.  We will identify
the kernel as 
the ordinary $5$-strand braid group, and then 
work out the details of the
group extension.

We defined $\HDM\sset\BDM$ as the  union of the mirrors
of the roots of~$\LDM$.  In lemma~\ref{LemMirrorsMeetingBDM} we showed that these
are exactly the roots of~$L$ whose mirrors meet~$\BDM$,
except for those orthogonal to~$\BDM$.
This immediately implies:

\begin{lemma}
    \label{LemNeighborhoodU}
    There is a $\PGDMsw$-invariant neighborhood $U$ of~$\BDM$,
    such that orthogonal projection $\pi:\BB^{13}\to\BDM$ realizes
    $U-\H$ as a fibration over $\BDM-\HDM$, with fibers as follows.
    The fiber over 
    each $x\in\BDM-\HDM$ is an open 
    ball centered at~$x$, in the $\BB^4$ orthogonal
    to $\BDM$ at~$x$, minus  the $40$ mirrors of
    $L_4$.
    \qed
\end{lemma}

Now we choose a basepoint
$\sigma\in\overline{\tau\rho}-\{\rho\}$,
close enough to~$\rho$ so that the meridians $\sigma_j$
defined in section~\ref{SecChange} lie in~$U$, for $j=0,\dots,11,\rmA,\dots,\rmD$.
By moving $\sigma$ closer to~$\rho$ we may 
also suppose 
that $\sigma$'s 
$\PG$-stabilizer is a subgroup of~$\rho$'s.
We set
$$
J:=\piorb\bigl((U-\H)/\PGDMsw,\sigma\bigr)
$$
Most of this section concerns relations between 
$\sigma_0,\dots,\sigma_{11},\sigmaA,\dots,\sigmaD$.
Because the roots $s_0,\dots,s_{11},
\sA,\dots,\sD$
form 
an $\Atilde_{11}A_4$ subdiagram of $P^2\F_3$, the corresponding
meridians $\tau_0,\dots,\tau_{11},
\tauA,\dots,\tauD$
(all based at~$\tau$)
satisfy the 
$\Atilde_{11}A_4$ Artin relations.  
This is theorem~4.4 of \cite{Basak-bimonster-2}.
By  theorem~\ref{ThmBasepointChange},
$\sigma_0,\dots,\sigma_{11},
\sigmaA,\dots,\sigmaD$
do too.

Orthogonal projection $\BB^{13}\to\BDM$ is $\PGDMsw$-equivariant,
inducing a map 
\begin{equation}
    \label{EqProjectionToSubballQuotient}
    (U-\H)/\PGDMsw\to(\BDM-\HDM)/\PGDM
\end{equation}
Here $\PGDM$ is the group of automorphisms of $\BDM$ induced
by $\PGDMsw$.  That is, $\PGDM$ is the quotient of $\PGDMsw$ by
the pointwise stabilizer of $\BDM$.  This agrees with our previous
use of $\PGDM$ to mean the projective isometry group of~$\LDM$, because
$\LDM$ is a summand of~$L$.  The projection \eqref{EqProjectionToSubballQuotient}
induces a homomorphism
\begin{equation}
    \label{EqProjection}
    J
    \to\piorb\bigl((\BDM-\HDM)/\PGDM,\rho\bigr)
    \iso
    \piorb(\M_{12},T)
\end{equation}
The isomorphism here uses our refinement (theorem~\ref{ThmDeligneMostow}) of the theorem of
Deligne and Mostow.
It is well-known that $\M_{n}$ has contractible orbifold
universal  cover \cite{Birman}.  Therefore, applying the long exact homotopy
sequence to
the fibration in 
lemma~\ref{LemNeighborhoodU} yields an exact sequence
\begin{equation}
    \label{EqExactSequenceBefore}
    1\to
    \piorb(\hbox{fiber over~$\rho$},\sigma)
    \to J
    \to\piorb(\M_{12},T)
    \to 1.
\end{equation}
Next we
work out the left term.

\begin{lemma}
    \label{LemKernel}
    The kernel of \eqref{EqExactSequenceBefore} is isomorphic to the
    $5$-strand braid group $\Br_5$.
    The meridians 
    $\sigmaA,\dots,\sigmaD$
    form a standard
    set of generators.  That is, their $A_4$ Artin relations are 
    defining relations.
\end{lemma}

\begin{proof}
    We recall that $\Aut(L_4)$ 
    is the finite complex reflection
    group numbered $32$ in the Shephard-Todd list \cite[Table VII]{ShephardTodd}.  
    The fibration in lemma~\ref{LemNeighborhoodU}
    shows that the kernel of \eqref{EqExactSequenceBefore} is
    \begin{equation}
        \label{EqE8EBraidSpace}
        \piorb\bigl((B-\H)/\!\Aut(L_4),\,\sigma\bigr)
    \end{equation}
    where $B$ is 
    a small ball centered at~$\rho$, 
    in the orthogonal complement to $\BDM$ at~$\rho$.
    This is essentially the definition
    of 
    the braid group associated to $\Aut(L_4)$.
    (The standard definition \cite{Bessis} uses $\C^4$ in place of~$B$,
    and $\pi_1$ in place of $\piorb$, which is no change at all
    because finite complex reflection groups act freely on their
    mirror complements.)
    Orlik and Solomon showed that this
    is isomorphic to the standard $5$-strand braid group \cite[Theorem 2.25]{OS}. 

    It remains to show that 
    $\sigmaA,\dots,\sigmaD$ are a standard
    set of generators.  We use \cite[Thm.~1.2]{AB-braidlike} to prove
    generation.  This requires us to check several things.  First,
    the reflections in 
    $\sA,\dots,\sD$ generate $\Aut(L_4)$,
    which is well-known.  Second: of the $40$ mirrors
    of $\Aut(L_4)$, the ones closest to $\sigma$ are the mirrors 
    orthogonal to these four roots.  This is an easy computer check.
    To state the third condition,
    we write $E$  for the set of eight images of $\sigma$ under the $\w$-
    and $\wbar$-reflections in these four mirrors.  One must check,
    for each of the remaining $36$ mirrors, that the projection of $\sigma$ 
    to that mirror is closer to some element of~$E$ than it is to~$\sigma$.
    This is another easy computer check.

    The $A_4$ Artin relations satisfied by 
    $\sigmaA,\dots,\sigmaD$ 
    are the
    relations defining $\Br_5$ in terms of its standard generators.  
    So we obtain a self-surjection $\Br_5\to\Br_5$ by sending
    some standard set of generators to 
    $\sigmaA,\dots,\sigmaD$.
    The braid group is Hopfian \cite[p.\ 276]{Bell-Margalit}, which means that 
    any self-surjection must be an automorphism.  Therefore 
    $\sigmaA,\dots,\sigmaD$
    are themselves a standard set of generators.
\end{proof}

\begin{remark}
    We sketch a geometric argument that avoids the Hopfian trickery
    and the reliance on \cite{OS}.  First one proves
    generation, as above, and it remains to show that the braid
    relations on
    $\sigmaA,\dots,\sigmaD$
    define the subgroup of~$J$ that 
    they generate.
    Suppose $\EuScript{R}(\sigmaA,\dots,\sigmaD)=1$ is a relation they satisfy.
    By theorem~\ref{ThmBasepointChange} we also have
    $\EuScript{R}(\tauA,\dots,\tauD)=1$.  Then by $L_3(3)\semidirect2$ symmetry it follows
    that $\EuScript{R}(\tau_0,\dots,\tau_3)=1$.  Another application of
    theorem~\ref{ThmBasepointChange} shows $\EuScript{R}(\sigma_0,\dots,\sigma_3)=1$,
    and then projecting to $(\BDM-\HDM)/\PGDM$ gives $\EuScript{R}(\rho_0,\dots,\rho_3)=1$.
    Using normal forms
    for words in braid groups, one can
    show: the homomorphism 
    $\Br_5(\C)\to\Br_{12}(\cpone)\to\piorb(\M_{12})$,
    sending the standard generators of $\Br_5$ to 
    four consecutive standard generators of $\Br_{12}(\cpone)$,
    is injective.  It follows that $\EuScript{R}(\rho_0,\dots,\rho_3)=1$
    is a consequence of the braid relations on $\rho_0,\dots,\rho_3$.
    Reversing the first part of the argument shows that
    the same holds with $\sigmaA,\dots,\sigmaD$ in place of
    $\rho_0,\dots,\rho_3$.
\end{remark}

We have established the exact
sequence
\begin{equation}
    \label{EqExactSequence}
1\to\Br_5
\to
J
\to
\piorb(\M_{12})\to 1
\end{equation}
We know that $J$ is generated by 
$\sigma_0,\dots,\sigma_{11},\sigmaA,\dots,\sigmaD$. 
Three kinds of relations suffice to define~$J$.  First are the
relations defining the normal subgroup $\Br_5$, which are words
in $\sigmaA,\dots,\sigmaD$.  Then there are relations saying
how $\sigma_0,\dots,\sigma_{11}$ act on this normal subgroup
(they centralize it).  Finally there are relations of the form
\begin{equation}
    \label{EqWords}
\hbox{word}(\sigma_0,\dots,\sigma_{11})=
    \hbox{word}(\sigmaA,\dots,\sigmaD)
\end{equation}
where the left side is one of the relators defining
$\piorb(\M_{12})$, except written with $\sigma_j$'s
in place of $\rho_j$'s.  And the right side is a word in
$\sigmaA,\dots,\sigmaD$ that is equal to the left side word.
(Such a word exists,  because the left side lies in the
kernel of the projection to $\piorb(\M_{12})$).  It is easy to
see that any word in $\sigma_0,\dots,\sigma_{11},\sigmaA,\dots,\sigmaD$,
that is trivial in~$J$, can be reduced to the trivial word
by use of these relators.  So a set of defining relations
for~$J$ consists of 
the Artin relators involving any of $\sigmaA,\dots,\sigmaD$
(including their commutativity with $\sigma_0,\dots,\sigma_{11}$),
plus one relation of the form \eqref{EqWords} for each defining
relator of $\piorb(\M_{12})$.

It remains to work out the words \eqref{EqWords} explicitly.
As preparation, we record
some facts about important elements of $\PG$.
We write 
$S_0,\dots,S_{11},\SA,\dots,\SD$
for the $\omega$-reflections
in the roots
$s_0,\dots,s_{11},\sA,\dots,\sD$.

\begin{lemma}
    \label{LemSpecialWordsInSs}
    \leavevmode
    \begin{enumerate}
        \item
            \label{ItemIncreasingDecreasingSs}
            $S_jS_{j+1}\cdots S_{j+10}$ and $S_jS_{j-1}\cdots S_{j-10}$
            are independent of~$j\in\{0,\dots,11\}$. 
            Both of them, and 
            $\Delta(\SA,\dots,\SD)$, act on $\BB^1(\rho,\tau)$ 
            by the positive $\pi/6$ rotation around~$\rho$.
        \item
            \label{ItemScalarsOnLDM}
            The actions of 
            $S_1S_2\cdots S_{10}S_{11}^2 S_{10}\cdots S_2S_1$
            and 
            $\Delta(\SA,\dots,\SD)^2$ 
            on $\BB^{13}$ coincide.  Namely, both
            $e^{\pi i/3}S_1\cdots S_{10}S_{11}^2S_{10}\cdots S_1$
            and $\Delta(\SA,\dots,\SD)^2$ act trivially on $\LDM$ and
            by the scalar $e^{\pi i/3}$ on its orthogonal complement.
        \item
            \label{ItemDeltaOfSs}
            $\Delta(S_1,\dots,S_{11})$ acts on
            $\BB^1(\rho,\tau)$
            by the $\pi$ rotation around~$\rho$.
    \end{enumerate}
\end{lemma}

\begin{proof}
    Computation.
\end{proof}

\begin{lemma} 
    \label{LemStabilizerOfSigma}
    The $\PG$-stabilizer of $\sigma$ fixes
    $\BB^1(\rho,\tau)$ pointwise, has generators
    \begin{center}
        \begin{tabular}{cl}
            $S_1 S_2\cdots S_{11}\cdot\Delta(\SA,\dots,\SD)^{-1}$%
            &of order~$12$ and
            \\
            $(S_1 S_2\cdots S_{11})^6\cdot\Delta(S_1,\dots,S_{11})^{-1}$%
            &of order~$2$,
        \end{tabular}
    \end{center}
    and is dihedral of order~$24$.
    Conjugation by the second word inverts the first.
\end{lemma}

\begin{proof}
    We chose $\sigma$ close enough to $\rho$ so that $\sigma$'s stabilizer
    lies in $\rho$'s.  Therefore $\sigma$'s stabilizer acts trivially
    on~$\BB^1(\rho,\sigma)=\BB^1(\rho,\tau)$.
    The stabilizer of $\tau$ is $L_3(3)\semidirect2$, and its subgroup fixing~$\rho$
    (or equivalently~$\sigma$)
    is dihedral of order~$24$.  That the listed elements stabilize~$\sigma$
    follows from lemma~\ref{LemSpecialWordsInSs}.   
    Computation establishes their orders and that the second inverts the first.
    We remark that
    the exponent~$^{-1}$ could be removed in the second word, because
    $\Delta(S_1,\dots,S_{11})$ has order~$2$.  However, the current form matches
    the lift of this word with $\sigma_j$'s in place of $S_j$'s, in 
    lemma~\ref{LemLocalGroupAtSigma}.
\end{proof}

\begin{lemma}
    \label{LemSpecialPaths}
    Let $\gamma$ resp.\ $\delta$ be the positively-directed
    circular arc in $\BB^1(\rho,\tau)$,
    beginning at~$\sigma$, and subtending angle $\pi/6$
    resp.\ $\pi/2$ around its center~$\rho$.
    Then 
    \begin{enumerate}
        \item
            \label{ItemIncreasingPath}
            the $\sigma_j\sigma_{j+1}\cdots\sigma_{j+10}$ are
            equal to each other and to $(\gamma,S_jS_{j+1}\cdots S_{j+10})$;
        \item
            \label{ItemDecreasingPath}
            the $\sigma_j\sigma_{j-1}\cdots\sigma_{j-10}$
            are equal to each other and to $(\gamma,S_{j}S_{j-1}\cdots S_{j-10})$;
        \item
            \label{ItemFundamentalPathAtoD}
            $\Delta(\sigmaA,\dots,\sigmaD)$ is equal to
            $(\gamma,\Delta(\SA,\dots,\SD))$.
        \item
            \label{ItemFundamentalPath1to11}
            $\Delta(\sigma_1,\dots,\sigma_{11})$ is equal to 
            $(\delta,\Delta(S_1,\dots,S_{11}))$.
    \end{enumerate}
\end{lemma}

\begin{proof}[Proof of lemma~\ref{LemSpecialPaths}\eqref{ItemIncreasingPath}--\eqref{ItemDecreasingPath}]
    We will prove 
    $\sigma_1\cdots\sigma_{11}=(\gamma,S_1\cdots S_{11})$.  The rest of
    \eqref{ItemIncreasingPath}
    follows by symmetry.  Namely, the $\PG$-stabilizer of~$\sigma$
    preserves~$\gamma$, and by lemma~\ref{LemStabilizerOfSigma} contains an element that
    cyclically permutes the $s_i$ (hence the $\sigma_i$ and~$S_i$).
    This permutation leaves $(\gamma,S_1\cdots S_{11})$ invariant because
    $S_j S_{j+1}\cdots S_{j+10}$ is independent of~$j$ (lemma~\ref{LemSpecialWordsInSs}).
    By lemma~\ref{LemStabilizerOfSigma}, another element of $\PG$ stabilizes $\sigma$
    and reverses the order of the $s_i$, so
    \eqref{ItemDecreasingPath} follows from
    \eqref{ItemIncreasingPath}.

    We write $\beta$ for the path underlying
    $\sigma_1\cdots\sigma_{11}$, namely
    \begin{center}
    \begin{tabular}{lll}
        $\beta:=\mu_{\sigma,s_1}$
        &followed by
        &$S_1(\mu_{\sigma,s_2})$
        \\
        &followed by
        &$S_1S_2(\mu_{\sigma,s_3})$
        \\
        &followed by&$\ldots$
        \\
        &followed by
        &$S_1S_2\cdots S_{10}(\mu_{\sigma,s_{11}})$
    \end{tabular}
    \end{center}
    It is obvious that $\sigma_1\cdots\sigma_{11}$ has the same underlying
    element of $\PG$ as $(\gamma,S_1\cdots S_{11})$.  In particular,
    the endpoints of
    $\beta$ and $\gamma$ coincide.  So it remains to show
    that the loop $\beta\gamma^{-1}$ (ie, $\beta$ followed by  $\reverse(\gamma)$) 
    is nullhomotopic in~$U-\H$.  In fact we will show that it is nullhomotopic
    in $V-\H$, where $V$ is the intersection of~$U$ with 
    the $\BB^{10}$ spanned by $\BDM$ 
    and $\tau$.  It follows from lemma~\ref{LemNeighborhoodU}  that
    $V-\H$ is a punctured-disk
    bundle over
    $\BDM-\HDM$.  

    Our first step is to show that the projection of $\beta\gamma^{-1}$
    to $\BDM-\HDM$ is nullhomotopic.   To do this we find braids representing
    the projections of $\sigma_1\cdots\sigma_{11}$ and $(\gamma,S_1\cdots S_{11})$ to
    $\piorb((\BDM-\HDM)/\PGDM,\rho)=\piorb(\M_{12},T)$.  By theorem~\ref{ThmDeligneMostow}, the former
    is represented by the braid $\rho_1\cdots\rho_{11}$.  The latter 
    projects to
    $$
    \bigl((\hbox{constant path at~$\rho$}),S_1\cdots S_{11}\bigr).
    $$
    This is represented by a braid which keeps all~$12$ points
    evenly spaced on the unit circle (because its underlying path
    in the moduli space is constant), and permutes them in the same way as
    $\rho_1\cdots\rho_{11}$ (because the element of $\PGDM$
     underlying $\rho_1\cdots\rho_{11}$ is $S_1\cdots S_{11}$).
    That is, $(\gamma,S_1\cdots S_{11})$ is represented by
    the braid in which all~$12$ points of~$T$ 
    move $\pi/6$ clockwise  around the unit circle with
    uniform speed. 
    The first braid, followed by the inverse of the second,
    is trivial in $\Br_{12}(\cpone)$.  It follows that the projection
    of $\beta\gamma^{-1}$ to $\BDM-\HDM$ represents the trivial element
    of $\piorb(\M_{12})$ and is therefore nullhomotopic.  

    Now the exact sequence on $\pi_1$, for the punctured-disk fibration
    $V-\H\to\BDM-\HDM$, shows that
    $\beta\gamma^{-1}$
    is 
    homotopic into the fiber over~$\rho$.  
    So it will be enough to show that $\beta\gamma^{-1}$
    has trivial winding number around~$\BDM$ in $\BB^{10}$.
    For this, we may replace $\beta$ by its projection $\alpha$
    to $\BB^1(\rho,\tau)$.  In summary: it will be enough to prove that
    the specific loop $\alpha\gamma^{-1}$ is nullhomotopic in 
    the specific punctured disk $V\cap\BB^1(\rho,\tau)$.  This becomes
    obvious upon drawing it, which we now prepare to do.

\begin{figure}
\begin{tikzpicture}[scale=.2] 
    \begin{scope}[xshift=0cm]
        \tikzset{middlearrow/.style={
        decoration={markings,
            mark= at position 0.7 with {\arrow{#1}} ,
        },
        postaction={decorate}
    }}
\draw[lightgray, ->] (0,0) -- (50,0);
\draw[lightgray, ->] (0,0) -- (43.3012702, - 25);
\draw[fill] (0, 0) circle [radius=0.1];
        \node [below] at (0,0) {$0$};
\draw[thick,orange]
(44.7846096908265, 0) --
(45.3619599600162, 0) --
(45.6506350946110, -0.500000000000004) --
(46.7279853638006, -0.788675134594815) --
(47.0166604983954, -1.86602540378444) --
(48.3826859021798, -2.65470053837925) --
(48.3826859021798, -4.23205080756888) --
(49.7487113059643, -5.59807621135332) --
(49.2487113059643, -7.46410161513776) --
(50.3260615751539, -9.33012701892219) --
(49.2487113059643, -11.1961524227066) --
(49.8260615751539, -13.3508529610859);
        \draw[thick,orange,middlearrow={latex}] (49.8260615751539, -13.3508529610859) --
(48.2487113059643, -14.9282032302755);
    \draw[thick,orange]
    (48.2487113059643, -14.9282032302755) --
(48.2487113059643, -17.0829037686548) --
(46.3826859021798, -18.1602540378444) --
(45.8826859021798, -20.0262794416288) --
(44.0166604983954, -20.5262794416288) --
(43.2279853638006, -21.8923048454133) --
(41.6506350946110, -21.8923048454133) --
(40.8619599600162, -22.6809799800081) --
(39.7846096908265, -22.3923048454133) --
(39.2846096908265, -22.6809799800081) --
(38.7846096908265, -22.3923048454133);
\draw[fill, red] (44.7846096908265, 0) circle [radius=0.1];
\draw[fill, blue]  (45.3619599600162, 0) circle [radius=0.1];
\draw[fill, red] (45.6506350946110, -0.500000000000004)  circle [radius=0.1];
\draw[fill, blue]  (46.7279853638006, -0.788675134594815)  circle [radius=0.1];
\draw[fill, red] (47.0166604983954, -1.86602540378444)  circle [radius=0.1];
\draw[fill, blue]  (48.3826859021798, -2.65470053837925)  circle [radius=0.1];
\draw[fill, red] (48.3826859021798, -4.23205080756888)  circle [radius=0.1];
\draw[fill, blue]  (49.7487113059643, -5.59807621135332)  circle [radius=0.1];
\draw[fill, red] (49.2487113059643, -7.46410161513776)  circle [radius=0.1];
\draw[fill, blue]  (50.3260615751539, -9.33012701892219)  circle [radius=0.1];
\draw[fill, red] (49.2487113059643, -11.1961524227066)  circle [radius=0.1];
\draw[fill, blue]  (49.8260615751539, -13.3508529610859)  circle [radius=0.1];
\draw[fill, red] (48.2487113059643, -14.9282032302755)  circle [radius=0.1];
\draw[fill, blue]  (48.2487113059643, -17.0829037686548)  circle [radius=0.1];
\draw[fill, red] (46.3826859021798, -18.1602540378444)  circle [radius=0.1];
\draw[fill, blue]  (45.8826859021798, -20.0262794416288)  circle [radius=0.1];
\draw[fill, red] (44.0166604983954, -20.5262794416288)  circle [radius=0.1];
\draw[fill, blue]  (43.2279853638006, -21.8923048454133)  circle [radius=0.1];
\draw[fill, red] (41.6506350946110, -21.8923048454133)  circle [radius=0.1];
\draw[fill, blue]  (40.8619599600162, -22.6809799800081)  circle [radius=0.1];
\draw[fill, red] (39.7846096908265, -22.3923048454133)  circle [radius=0.1];
\draw[fill, blue]  (39.2846096908265, -22.6809799800081)  circle [radius=0.1];
\draw[fill, red] (38.7846096908265, -22.3923048454133)  circle [radius=0.1];
%
\node [below left] at (44.7846096908265, 0) {begin};
    \node [left] at (38.7846096908265, -22.3923048454133) {end};
\node at (9,-2.5) {$\pi/6$};
\end{scope}
\end{tikzpicture}
\caption{The piecewise linear path 
$t\mapsto\ip{\beta(t)}{\rho}$ in~$\C$, up to scaling.
Its reciprocal, scaled, is the path~$\alpha$ at the heart of the proof of
lemma~\ref{LemSpecialPaths}\eqref{ItemIncreasingPath}.}
\label{fig-orange-path}
\end{figure}
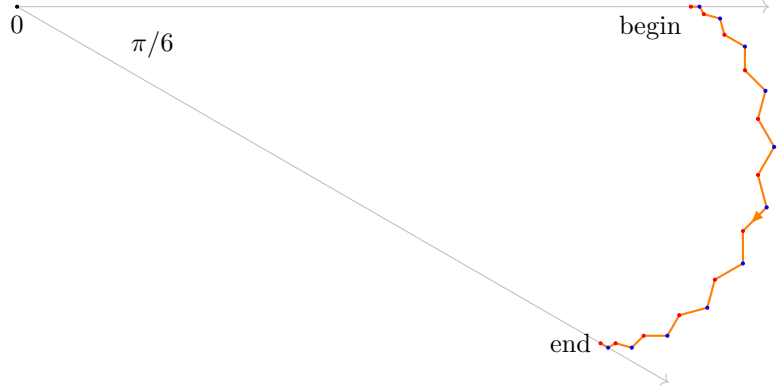


    First, each 
    $\mu_{\sigma,s_i}$ was defined as a perturbation of
    the concatenation of two geodesic segments.  The first segment joins
    $\sigma$ to the projection of $\sigma$ into $s_i^\perp$, which the second 
    segment joins to $S_i(\sigma)$.  The perturbation is to avoid
    hitting $s_i^\perp$ at the concatenation point, and can be made
    arbitrarily small.  The concatenation point does not lie in~$\BDM$, so the
    perturbation does
    not affect the  winding number of $\beta\gamma^{-1}$ around $\BDM$.
    So, for the rest of the proof, we will ignore the perturbation
    and regard $\mu_{\sigma,s_i}$
    as the concatenation of the two segments.   (Or you can imagine
    that the perturbation is so small that it becomes invisible in 
    figure~\ref{fig-orange-path}.)  This makes $\beta$
    into a concatenation of $22$ specific geodesic segments in $\BB^{10}$. 

    Second, if two vectors of negative norm in~$\C^{14}$ have
    negative inner product, then the segment joining them in~$\C^{14}$ 
    represents the geodesic joining the corresponding points in~$\BB^{13}$.
    Therefore $\mu_{\sigma,s_i}$ is represented by the line segment in~$\C^{14}$
    from  $\sigma$ to its linear projection
    onto $s_i^\perp$, followed by the line segment from there to $S_i(\sigma)$.
    In this way, we may regard $\beta$ as path in $\C^{14}$.  

    We orthogonally project  $\beta$ into
    $\C\langle\rho,\tau\rangle$ and decompose the result 
    as a linear combination of $\rho$ and $\tau-\rho$.  These are
    orthogonal to each other, with norms 
     $-36-24\sqrt3$ and
    $30+16\sqrt3$ respectively.  Therefore 
    $$
    \hbox{the projection of $\beta(t)$ to
    $\C\langle\rho,\tau\rangle$}=
    \frac{\ip{\beta(t)}{\rho}}{-36-24\sqrt3}\,\rho
    +\frac{\ip{\beta(t)}{\tau-\rho}}{30+16\sqrt3}(\tau-\rho)
    $$
    The coefficient of the second component is constant in time,
    because $\beta(t)$ differs from $\beta(0)=\sigma$ by a linear combination
    of $s_0,\dots,s_{11}$.  These are orthogonal to $\tau-\rho$
    because  $\rho$ was defined as the projection of $\tau$ to their
    span.  The constant \emph{does} depend on our choice of 
    basepoint~$\sigma=\beta(0)$, and is nonzero 
    because $\sigma\in\overline{\rho\tau}-\{\rho\}$.  It approaches~$0$ 
    as the choice of $\sigma$ approaches~$\rho$.
    On the other hand, the coefficient of the first component depends on time,
    but does not depend on the choice of
    $\sigma\in\overline{\rho\tau}-\{\rho\}$.  This is because any two
    candidates for $\sigma$ differ by a multiple
    of $\tau-\rho$.    
    
     We identify $P\C\langle\rho,\tau\rangle$ with $\cpone=\C\cup\{\infty\}$
     by sending $u\rho+v(\tau-\rho)$ to $v/u$.  
     This identifies $\rho$ with $0$, and
     $\BB^1(\rho,\tau)$ with an open disk
     centered there. (If one cares, the radius is $\sqrt{-\rho^2/(\tau-\rho)^2}$.)
     Most importantly, 
     we get the simple formula
     $$\alpha(t)=\frac{\hbox{constant depending on~$\sigma$, which approaches~$0$
     as $\sigma\to\rho$}}{\ip{\beta(t)}{\rho}, \hbox{ which is independent of~$\sigma$}},$$
     where
     $\beta(t)$ is a concatenation of $22$ segments in $\C^{14}$, starting
     at~$\beta(0)=\sigma$.

     The path appearing in
     figure~\ref{fig-orange-path} is {\it not}~$\alpha$.  Rather, the figure shows the piecewise
     linear path 
     $t\mapsto\ip{\beta(t)}{\rho}$, up to scale.  
    It is clearly homotopic, rel endpoints, to a clockwise circular arc 
    of angle $\pi/6$ around~$0$.  To check that the picture does not deceive,
    it is enough to verify that  the
    endpoints of the segments lie in the right halfplane in~$\C$.
    (This is why we work with the piecewise linear path,
    rather than $\alpha$ itself.)
    Therefore $\alpha$ is homotopic, rel endpoints, to the 
    \textit{counter}clockwise
    circular arc of angle $\pi/6$ in~$\BB^1(\rho,\tau)$, starting at~$\sigma$
    and centered at~$\rho$.  This is the same path as $\gamma$, proving that
    $\alpha\gamma^{-1}$ is nullhomotopic.
\end{proof}


\begin{proof}[Proof of lemma~\ref{LemSpecialPaths}\eqref{ItemFundamentalPathAtoD}]
    First we claim that $(\gamma,\Delta(\SA,\dots,\SD))$
    has the same image in the abelianization $\Z$ of $\Br_5$
    as $\Delta(\sigmaA,\dots,\sigmaD)$.
    To see this, observe that its twelfth power is the circle
     in 
     $\BB^1(\rho,\tau)$ centered at~$\rho$.
     This loop encircles the $40$ mirrors once each.  Since a 
    loop around one of these mirrors is (conjugate to) the $3$rd 
    power of one of the standard braid generators, it follows that
    the image of
    $(\gamma,\Delta(\SA,\dots,\SD))$ in~$\Z$ is
    $\frac{1}{12}(40\cdot3)=10$ 
    times that of a standard generator.  On the other
    hand, $\Delta(\sigmaA,\dots,\sigmaD)$ is defined as
    a product of $10$ standard generators.  This proves our claim.

    Next we claim that conjugation by $(\gamma,\Delta(\SA,\dots,\SD))$
    exchanges $\sigmaA\leftrightarrow\sigmaD$ and
    $\sigmaB\leftrightarrow\sigmaC$.  That is,
    it permutes the generators
    of $\Br_5$ in the same way that $\Delta(\sigmaA,\dots,\sigmaD)$
    does.  It follows that they differ by a central element of~$\Br_5$.
    It is well-known that $Z(\Br_5)$ is cyclic, generated by
    $\Delta(\sigmaA,\dots,\sigmaD)^2$.  Since the center
    maps faithfully to the abelianization, the previous paragraph
    shows that $\Delta(\sigmaA,\dots,\sigmaD)$ and
    $(\gamma,\Delta(\SA,\dots,\SD))$ are actually equal,
    not just equal up to center.

    It remains to prove the claim.  We 
    will show 
    \begin{equation}
        \label{EqPathConjugation}
    (\gamma,\Delta(\SA,\dots,\SD))\,\sigmaD\,(\gamma,\Delta(\SA,\dots,\SD))^{-1}=\sigmaA;
    \end{equation}
    the  same argument proves that $(\gamma,\Delta(\SA,\dots,\SD))$ also sends
    $\sigmaA$ to $\sigmaD$, and exchanges
    $\sigmaB\leftrightarrow\sigmaC$.  The element of $\PG$ underlying the left
    side of \eqref{EqPathConjugation} is 
    $$\Delta(\SA,\dots,\SD)\,\SD\,\Delta(\SA,\dots,\SD)^{-1}.$$
    Since $\Delta(\SA,\dots,\SD)$ exchanges $\SA\leftrightarrow\SD$, this
    simplifies to $\SA$.  This agrees with the element of $\PG$
    underlying the right side of \eqref{EqPathConjugation}.

    To simplify analysis of the underlying path, we 
    write $F_t$ for the $1$-parameter subgroup of $\U(13,1)$ 
    that acts trivially on
    $\LDM\tensor\C$ and by the scalar 
    by $e^{\pi i t/3}$ on its orthogonal complement.  
    On every complex line through~$\rho$, that is orthogonal to
    $\BDM$, it acts by $2\pi t/3$ rotation.  In particular, 
    $\gamma(t)=F_t(\sigma)$.  
    Also, although the $F_t$ do not all preserve~$\H$, they do preserve
    $U\cap\H$, which is the only part of~$\H$ that will be important in this
    proof.

    We also recall that the path underlying
    $\sigmaD$ is $\mu_{\sigma,\sD}$, moving from $\sigma$ to very near
    $\sD^\perp$, then $2\pi/3$ of the way around $\sD^\perp$, and 
    then onward to $\SD(\sigma)$.  
    What remains is to consider the path underlying the left side of
    \eqref{EqPathConjugation}.
    This is 
    $\gamma$, followed by the $\Delta(\SA,\dots,\SD)$-image
    of $\mu_{\sigma,\sD}$,
    followed by the $\Delta(\SA,\dots,\SD)\SD$-image of 
    $\reverse(\gamma)$.  The first part is
    the path $t\mapsto F_t(\sigma)$ as $t$ varies over $[0,1]$.  
    Because $\Delta(\SA,\dots,\SD)$
    exchanges $\sA^\perp$ with $\sD^\perp$, the second part
    is $\mu_{F_1(\sigma),\sA}$, from $F_1(\sigma)$
    to $\SA(F_1(\sigma))=F_1(\SA(\sigma))$.  
    The third part  is $t\mapsto F_{1-t}(\SA(\sigma))$.  We must show that this
    $3$-part path, followed by $\reverse(\mu_{\sigma,\sA})$,
    bounds a disk in $U-\H$.  
    This is almost obvious: consider the surface swept out
    by the $F_t$-images of  $\mu_{\sigma,\sA}$ as $t$ varies
    from $0$ to~$1$ (or equivalently, the surface swept out by
    the paths $\mu_{F_t(\sigma),\sA}$).  
\end{proof}

Before we prove the last part of lemma~\ref{LemSpecialPaths}, we remark that
we do not need it to describe~$J$ in theorem~\ref{ThmPi1ofUmodStabilizer} below.
It is needed 
for the second generator in our 
 description of
the local group at~$\sigma$ in lemma~\ref{LemLocalGroupAtSigma}, and hence
to establish the second relation in~$G$ in theorem~\ref{ThmNewRelations}.
However, that relation
is not needed in the proof of the main theorem.  So the reader may
be inclined to skip it.

\begin{proof}[Proof of lemma~\ref{LemSpecialPaths}\eqref{ItemFundamentalPath1to11}]
    In the ordinary $12$-strand
    braid group, $\Delta(\sigma_1,\dots,\sigma_{11})$
    can be described by saying that the $12$ strands
    are embedded in a strip that rotates about its midline through an angle
    $\pi$.  This represents a
    constant path in~$\M_{12}$, so $\Delta(\sigma_1,\dots,\sigma_{11})$
    is homotopic, rel endpoints, into the punctured disk which is the
    fiber of $V-\H$ over~$\rho$.  By 
    lemma~\ref{LemSpecialWordsInSs}\eqref{ItemDeltaOfSs}, its final
    endpoint is halfway around this fiber from $\sigma$.  To finish
    the proof, it is enough to prove that $\Delta(\sigma_1,\dots,\sigma_{11})^2$
    encircles the puncture $\rho$ once positively.  This follows
    from its equality with $(S_1\cdots S_{11})^{12}$ in $\Br_{12}$,
    and our description of $I$ in
    part \eqref{ItemIncreasingPath} of this lemma.
\end{proof}

\begin{theorem}
    \label{ThmPi1ofUmodStabilizer}
    The orbifold fundamental group
    $J=\piorb\bigl((U-\H)/\PGDMsw,\,\sigma\bigr)$ has generators
    $\sigma_0,\dots,\sigma_{11},\sigmaA,\dots,\sigmaD$ and
    defining relations
    \begin{enumerate}
        \item
            \label{EqAtilde11A4Relations}
            the Artin relations of the $\Atilde_{11} A_4$ diagram; 
        \item
            \label{EqEqualityOfIs}
            all the $I_j:=\sigma_j\sigma_{j+1}\cdots\sigma_{j+10}$ coincide; write $I$ for their common value;
        \item
            \label{EqEqualityOfDs}
            all the $D_j:=\sigma_j\sigma_{j-1}\cdots,\sigma_{j-10}$ coincide; write $D$ for their common value;
        \item
            \label{ItemIDisWord}
            $ID=\Delta(\sigmaA,\dots,\sigmaD)^2$;
        \item
            \label{ItemD6isI6}
            $D^6=I^6$.
    \end{enumerate}
    Furthermore, for all $k=0,\dots,11$ we have
        $I\sigma_k I^{-1}{}=\sigma_{k+1}$ and
        $D\sigma_k D^{-1}{}=\sigma_{k-1}$ and
        $\Delta(\sigma_1,\dots,\sigma_{11})\sigma_k\Delta(\sigma_1,\dotsm\sigma_{11})^{-1}
        {}=
        \sigma_{12-k}$
\end{theorem}

Here and henceforth we use $I_j,D_j,I,D$ for these words in the $\sigma_k$'s
rather than the $\rho_k$'s.  To avoid confusion, we will not use these
symbols in the proof until after establishing 
\eqref{EqAtilde11A4Relations}--\eqref{ItemD6isI6}.

\begin{proof}
    We will
    establish the stated relations; that they are defining relations
    follows.  This
    is because they include defining relations for the normal
    subgroup $\Br_5$
    generated by $\sigmaA,\dots,\sigmaD$, and relations saying
    how $\sigma_0,\dots,\sigma_{11}$ conjugate $\sigmaA,\dots,\sigmaD$,
    and also one relation  of the form \eqref{EqWords}, for each 
    defining relator of $\piorb(\M_{12})$ from 
    lemma~\ref{ThmPi1ModuliSpace}.  
    (See the discussion before lemma~\ref{LemSpecialWordsInSs}.)
    Since $12$ is even, we may use the alternative form
    $(\rho_1\cdots\rho_{11})^6=(\rho_{11}\cdots\rho_1)^6$ 
    of lemma~\ref{ThmPi1ModuliSpace}\eqref{ItemM0nInEqualsDnEquals1}.
    The $\Atilde_{11}A_4$ Artin
    relations hold by \cite[Thm.\ 4.4]{Basak-bimonster-2}.
    Lemma~\ref{LemSpecialPaths} establishes 
    \eqref{EqEqualityOfIs} and \eqref{EqEqualityOfDs}.

    \eqref{ItemIDisWord}
    Lemma~\ref{LemSpecialWordsInSs} 
    shows that the elements
    of $\PG$ underlying 
    $$\sigma_1\cdots\sigma_{10}\sigma_{11}^2\sigma_{10}\cdots,\sigma_1
    \hbox{\ and\ }
    \Delta(\sigmaA,\dots,\sigmaD)^2,$$ 
    namely 
    $S_1\cdots S_{10}S_{11}^2S_{10}\cdots S_1$ and $\Delta(\SA,\dots,\SD)^2$,
    coincide.  
    Lemma~\ref{LemSpecialPaths} shows that both underlying paths 
    are homotopic, rel endpoints, to  
    the positive circular arc in $\BB^1(\rho,\tau)$
    that starts at~$\sigma$ and 
    subtends angle $\pi/3$ around its center~$\rho$.

    \eqref{ItemD6isI6}
    Lemma~\ref{LemSpecialPaths} shows that 
    both $(\sigma_1\cdots\sigma_{11})^{6}$ and
    $(\sigma_{11}\cdots\sigma_1)^{6}$ have the same underlying
    path, namely  the positive semicircular arc 
    in $\BB^1(\rho,\tau)$, starting at~$\sigma$ and centered at~$\rho$.
    And one can check that their underlying elements 
    $(S_1\cdots S_{11})^6$ and $(S_{11}\cdots S_1)^6$ of~$\PG$
    are equal.  
    
    Because $\sigma_0,\dots,\sigma_{11}$ satisfy \eqref{EqEqualityOfIs} and
    the $\Atilde_{11}$ Artin relations,
    theorem~\ref{ThmBraidGroups}\eqref{ItemAdjoinInftyRelations} shows that
    $I\sigma_{k}I^{-1}=\sigma_{k-1}$
    for all~$k$.  And similarly with $D$ in place of~$I$.   
    For $k\neq0$ in the final relation, the $A_{11}$ Artin relations
    satisfied by~$\sigma_1,\dots,\sigma_{11}$ are enough: this
    is one of the standard properties of the fundamental element
    of $\Br_{12}$.
    To check that it also holds for $k=0$, 
    one can just observe that
    $\sigma_k\mapsto\sigma_{12-k}$ 
    and conjugation by $\Delta(\sigma_1,\dots,\sigma_{11})$
    are automorphisms of~$J$
    that 
    agree on a set of generators (since \eqref{EqEqualityOfIs}
    allows
    $\sigma_0$ to be expressed in terms of $\sigma_1,\dots,\sigma_{11}$).
    %
%
\end{proof}

\section{Proof of the main theorem}
\label{SecMain}

\noindent
We will find new relations in~$G$
by using the previous section to write down generators for the local group
at~$\sigma$, defined in section~\ref{subsec-orbifold-fundamental-group}, 
as words in the meridians 
$\sigma_0,\dots,\sigma_{11},\sigmaA,\dots,\sigmaD$.  Because the
$\PG$-stabilizer of $\sigma$ lies inside the $L_3(3)\semidirect2$ stabilizing~$\tau$,
conjugation by these words must permute the point- and line-meridians,
ie the $26$ Artin generators of~$G$.
The new relations express this conjugation action.
We remark that we may speak of ``the'' local group at~$\sigma$,
even though
we have considered $\sigma$ as a point of two different orbifolds,
$(U-\H)/\PGDMsw$ and $(\BB^{13}-\H)/\PG$.  Their local groups at
$\sigma$ coincide because 
the full 
$\PG$-stabilizer of~$\sigma$
preserves~$\BDM$, hence $U$.  
In the next lemma, $I$ means the ``increasing'' product
$\sigma_j\sigma_{j+1}\cdots\sigma_{j+10}$, which is independent
of~$j$ by lemma~\ref{LemSpecialPaths}. 

\begin{lemma}
    \label{LemLocalGroupAtSigma}
    The local group at $\sigma$ is  dihedral
    of order~$24$, generated by 
    \begin{center}
        \begin{tabular}{cl}
            $I\cdot \Delta(\sigmaA,\dots,\sigmaD)^{-1}$%
            &of order~$12$ and
            \\
            $I^6\cdot\Delta(\sigma_1,\dots,\sigma_{11})^{-1}$%
            &of order~$2$.
        \end{tabular}
    \end{center}
    The second inverts the first.
\end{lemma}

\begin{proof}
    Parts \eqref{ItemIncreasingPath} and
    \eqref{ItemFundamentalPathAtoD} of
    Lemma~\ref{LemSpecialPaths} shows that the path $\gamma$ (defined
    there) underlies $I=\sigma_1\cdots\sigma_{11}$ 
    and also $\Delta(\sigmaA,\dots,\sigmaD)$.  Furthermore,
    $\gamma$ lies in  $\BB^1(\rho,\tau)$, on which
    $S_1\cdots S_{11}$ acts by rotating
    $\gamma(0)$ to $\gamma(1)$.  Therefore the path
    underlying
    $I\cdot\Delta(\sigmaA,\dots,\sigmaD)^{-1}$
    is homotopic to  $\gamma$ 
    followed by the reverse of~$\gamma$.
    So 
    $I\cdot\Delta(\sigmaA,\dots,\sigmaD)^{-1}$
    represents the same element of~$J$ as
    $$
    \bigl(\hbox{the constant path at~$\sigma$},\,
    S_1S_2\cdots S_{11}\cdot\Delta(\SA,\dots,\SD)^{-1}\bigr)
    $$
    This
    gets identified with the first element of $\PG$
    listed in lemma~\ref{LemStabilizerOfSigma},
    under the correspondence between the $\PG$-stabilizer
    and the local group of $\sigma$.
    A similar argument using parts 
    \eqref{ItemIncreasingPath} and \eqref{ItemFundamentalPath1to11}
    of lemma~\ref{LemSpecialPaths} identifies
    $I^6\cdot\Delta(\sigma_1,\dots,\sigma_{11})^{-1}$
    with the second element listed there.
\end{proof}

We know from 
\cite{AB-26gens}
that the point- and line-meridians generate~$G$, and from
\cite[Thm.\ 4.4]{Basak-bimonster-2} that they satisfy
the Artin relations of $P^2\F_3$.  Therefore
$G$ is a quotient of $\Art(P^2\F_3)$.
But for current purposes it
is cleaner to express it as a quotient of
$\Art(P^2\F_3)\semidirect\Aut(P^2\F_3)$, where
the Artin generators map to 
the point-
and line-meridians, and $\Aut(P^2\F_3)=L_3(3)\semidirect2$ is identified 
with the local group at~$\tau$.

\begin{theorem}[New relations in~$G$]
    \label{ThmNewRelations}
    The orbifold fundamental group
    $$G=\piorb\bigl((\BB^{13}-\H)/\PG,\,\tau\bigr)$$
    is the quotient
    of $\Art(P^2\F_3)\semidirect\Aut(P^2\F_3)$ by the following
    relations and possibly some additional (presently
    unknown) relations.

    Consider any $\Atilde_{11}A_4$ subdiagram of $P^2\F_3$, and 
    any labeling of the nodes of $\Atilde_{11}$ by 
    $\tau_0,\dots,\tau_{11}$, cyclically around it
    in either direction, and either labeling of the
    nodes of $A_4$ by $\tauA,\dots,\tauD$ along it.
    Then
    \begin{enumerate}
        \item
            \label{ItemIRelationInG}
            $\tau_1\tau_2\cdots \tau_{11}\cdot\Delta(\tauA,\dots,\tauD)^{-1}$ equals
            the element of $\Aut(P^2\F_3)$ that permutes the
            nodes of the $\Atilde_{11}$ by $\tau_j\mapsto \tau_{j+1}$, and those of
            the $A_4$ by $\tauA\leftrightarrow\tauD$ and $\tauB\leftrightarrow\tauC$.
        \item
            \label{ItemIto6thRelationInG}
            $(\tau_1\tau_2\cdots \tau_{11})^{6}\cdot\Delta(\tau_1,\dots,\tau_{11})^{-1}$
            equals the element of $\Aut(P^2\F_3)$ that permutes the nodes
            of the $\Atilde_{11}$ by $\tau_j\mapsto \tau_{6-j}$, and fixes each
            of $\tauA,\dots,\tauD$.
    \end{enumerate}
\end{theorem}

\begin{proof}
    Recall from section~\ref{SecChange}
    the $16$ particular point-
    and line-meridians $\sigma_0,\dots,\sigma_{11},\discretionary{}{}{}\sigmaA,\dots,\sigmaD$
    based at~$\sigma$, and the corresponding meridians
    $\tau_0,\dots,\tau_{11},\tauA,\dots,\tauD$
    based at~$\tau$.  These are $16$
    of the Artin generators, forming a labeled $\Atilde_{11}A_4$ diagram as
    specified in the statement.  

    The segment $\overline{\sigma\tau}$ identifies the orbifold
    fundamental groups of $(\BB^{13}-\H)/\PG$ based at $\sigma$
    and $\tau$.  It is fixed pointwise by the $\PG$-stabilizer
    $D_{24}$ of~$\sigma$, and therefore identifies the local group
    at~$\sigma$ with 
    with a subgroup of~$L_3(3)\semidirect2$.  By theorem~\ref{ThmBasepointChange}, this
    segment also identifies each 
    of $\sigma_0,\dots,\sigma_{11},\sigmaA,\dots,\sigmaD$
    with the corresponding 
    $\tau_0,\dots,\tau_{11},\tauA,\dots,\tauD$.  Therefore
    the words in lemma~\ref{LemLocalGroupAtSigma}, with $\sigma_j$'s replaced by $\tau_j$'s,
    permute the $26$ Artin generators.  
    The pointwise $(L_3(3)\semidirect2)$-stabilizer
    of an $\Atilde_{11}A_4$ diagram is trivial.  Therefore these words'
    actions on the $26$ generators are completely determined by
    their actions on $\tau_0,\dots,\tau_{11},\tauA,\dots,\tauD$.  

    We will work out this action in the case
    of the second word; the argument
    for the first is similar.  Every word in $\tau_1,\dots,\tau_{11}$
    centralizes $\tau_A,\dots,\tau_D$, so it is enough to show that
    $(\tau_1\cdots\tau_{11})^6\Delta(\tau_1,\dots,\tau_{11})^{-1}$
    conjugates $\tau_j$ to $\tau_{6-j}$ for each $j=0,\dots,11$.
    By theorem~\ref{ThmBasepointChange}, it is enough to prove this with
    all
    $\tau_k$'s replaced by $\sigma_k$'s, which is immediate from the
    identities at the end of theorem~\ref{ThmPi1ofUmodStabilizer}.

    We have established the lemma for one particular choice
    of $\tau_0,\dots\tau_{11},\tauA,\dotsm\tauD$.
    This does not quite prove the theorem, because $L_3(3)\semidirect2$
    acts with two orbits on the set of 
    such choices.  The other orbit is
    represented by the labeling 
    $\tau_0,\dots,\tau_{11},\discretionary{}{}{}\tauD,\discretionary{}{}{}\dots,\discretionary{}{}{}\tauA$. 
    But this case follows from the first, by
    the relation
    $\Delta(\tauA,\dots,\tauD)=\Delta(\tauD,\dots,\tauA)$
    in~$\Br_5$.
\end{proof}

\begin{proof}[Proof of theorem~\ref{ThmMain}]
    Consider
    $\delta:=\tau_0\tau_1\cdots\tau_{10}(\tau_1\tau_2\cdots\tau_{11})^{-1}$,
    regarded as an element of $\Art(\Atilde_{11})\sset\Art(P^2\F_3)$.   We call it the
    ``deflation word'' for reasons that will become clear.
    An immediate consequence of 
    theorem~\ref{ThmNewRelations}\eqref{ItemIRelationInG} is that 
    $\delta$ dies in~$G$.  
    (In fact, we established this early in
    the proof of lemma~\ref{LemSpecialPaths}\eqref{ItemIncreasingPath},
    so the later parts of section~\ref{SecNeighborhood} are
    not needed for this theorem.)
    We write $\deltabar$ for the image of $\delta$ in $\Cox(\Atilde_{11})\sset\Cox(P^2\F_3)$.

    We claim that the subgroup of
    $\Cox(\Atilde_{11})\iso\Z^{11}\semidirect S_{12}$ normally
    generated
    by~$\deltabar$ is the translation subgroup $\Z^{11}$.
    This is standard: annihilating $\deltabar$ expresses $\tau_0$
    as a word in $\tau_1,\dots,\tau_{11}$, so the quotient can
    be no larger than $S_{12}$.  And the quotient is no smaller,
    because the transpositions $(1\ 2),(2\ 3),\dots,(11\ 12),(12\ 1)\in S_{12}$
    satisfy the relations of $\Cox(\Atilde_{11})$ and also the deflation
    relation.
    
    Now we can quote a result of Conway and Simons \cite{Conway-Simons},
    which depends essentially on work of Ivanov \cite{Ivanov} 
    and Norton \cite{Norton}.
    They consider the ``deflation'' of $\Cox(P^2\F_3)$, meaning the
    quotient of this group by the subgroup normally generated by
    the translation subgroups of the
    $\Cox(\Atilde_{11})$'s coming from all the $\Atilde_{11}$ subdiagrams of
    $P^2\F_3$.  They prove that this quotient is 
    the bimonster.  Because $G$ is a quotient of $\Art(P^2\F_3)$ in which
    $\delta$ dies,
    $G/S$ is a quotient of $\Cox(P^2\F_3)$ in which $\deltabar$
    dies.  By the previous paragraph, $G/S$ is a quotient of the 
    deflation of $\Cox(P^2\F_3)$, hence of  the bimonster $(M\times M)\semidirect2$.  
    Since $M$ is simple, the only possibilities for $G/S$ are
    the bimonster, $\Z/2$ and the trivial group.  The last case
    is excluded by lemma~\ref{LemGNontrivial} below.
\end{proof}

\begin{numberedremark}
    \label{RkHeckmanArgument}
    We sketch an argument of Heckman 
    \cite{Heckman}.  It lets one 
    skip most of sections \ref{SecDeligneMostow} and~\ref{SecNeighborhood}, and still 
    prove theorem~\ref{ThmMain}, although it does not establish
    any new relations in~$G$ itself.  Killing the squares of the meridians collapses
    the exact sequence \eqref{EqExactSequence}, namely
    $$
    1\to\Br_5\to J\to\piorb(\M_{12})\to1
    $$
    so that the left
    term becomes~$S_5$ and the right term~$S_{12}$.
    Because $S_5$ has no center, the extension is
    determined by the homomorphism $S_{12}\to\Out(S_5)=1$.
    Therefore $J$  becomes $S_5\times S_{12}$.
    The images of 
$\sigma_0,\dots,\sigma_{11}$ commute with $S_5$ and project
    to generators of~$S_{12}$.  Therefore 
    the subgroup of $S_5\times S_{12}$ they generate meets
    $S_5$ trivially and projects isomorphically to~$S_{12}$.  
    That is, the
    deflation relation holds for the images of
    $\sigma_0,\dots,\sigma_{11}$ in $G/S$.
    Then theorem~\ref{ThmBasepointChange} gives the same result
    for $\tau_0,\dots,\tau_{11}$, and one can quote Conway-Simons
    as in the proof above.  (Unfortunately,  Heckman's approach
    to theorem~\ref{ThmBasepointChange} had
    a gap.)
\end{numberedremark}

\begin{lemma}
    \label{LemGNontrivial}
    The quotient of $G=\piorb((\BB^{13}-\H)/\PG)$, by the subgroup~$S$
    generated by the squares of the meridians, is nontrival.
\end{lemma}

\begin{proof}
    It is enough to exhibit a connected
    orbifold double cover of~$(\BB^{13}-\H)/\PG$.
    There is a
    degree~$4$ holomorphic automorphic form $\Psi_0$ for~$\PG$, whose zero locus is~$\H$
    with multiplicity~$1$ along each component \cite[Thm.\ 7.1]{Allcock-Inventiones}.
    This means: $\Psi_0$ is a holomorphic function on the preimage  $\Omega$ of 
    $\BB^{13}$ in $\C^{14}-\{0\}$, 
    homogeneous of degree~$-4$, which has a simple zero along
    each component of the preimage $\Htilde$
    of~$\H$ in $\Omega$, and no other zeros.  
    The same reference shows that
    $\Psi_0$ transforms by a character $\Gamma\to\Z/3$.

    Let $Z$ be the space of pairs $(\ell,\phi)$, where $\ell$ is the
    preimage 
    in $\C^{14}-\{0\}$ of a point of $\BB^{13}-\H$, and $\phi$ is a
    holomorphic function on $\ell$ whose square coincides with $\Psi_0^3|_{\ell}$.
    This is a degree~$2$ covering space of $\BB^{13}-\H$ in an obvious way.
    It is connected because $\Psi_0^3$ has zeros of odd order along $\Htilde$.
    The $\Gamma$-invariance of $\Psi_0^3$ yields a
    $\Gamma$-action on~$Z$, with each $g\in\Gamma$ sending
    $(\ell,\phi)$ to $(g(\ell),\phi\circ g^{-1})$.  Because $\phi$ is homogeneous
    of degree~$-6$, the scalars in $\Gamma$ act trivially, yielding
    a $\PG$-action on~$Z$ and a $\PG$-equivariant map $Z\to\BB^{13}-\H$.
    Our promised double cover of $(\BB^{13}-\H)/\PG$ is $Z/\PG$.
\end{proof}

\end{document}